\documentclass{amsart}

\usepackage{amsfonts,amssymb,amsmath,enumerate,verbatim,mathtools,tikz,bm,mathrsfs,tikz-cd,hyperref,comment,stmaryrd,amsthm}
\usepackage{colonequals}

\usepackage{adjustbox}
\usepackage[all,pdf]{xy}
\hypersetup{
    colorlinks=true,
    linkcolor=blue,
    filecolor=magenta,      
    urlcolor=cyan,
}

\usepackage{cleveref}

\makeatletter
\@namedef{subjclassname@2020}{\textup{2020} Mathematics Subject Classification}
\makeatother

\author[Souvik Dey]{Souvik Dey}
\address{Department of Mathematical Sciences
850 West Dickson Street, University of Arkansas
Fayetteville, Arkansas 72701}
\email{souvikd@uark.edu}
\urladdr{\url{https://orcid.org/0000-0001-8265-3301}}

\author[Jian Liu]{Jian Liu}
\address{School of Mathematics and Statistics, and Hubei Key Laboratory of Mathematical Sciences,  Central China Normal University,  Wuhan 430079, P.R. China}
\email{jianliu@ccnu.edu.cn}
\urladdr{\url{https://orcid.org/0000-0001-8360-7024}}

\author[Yuki Mifune]{Yuki Mifune}
\address{Graduate School of Mathematics, Nagoya University, Furocho, Chikusaku, Nagoya 464-8602, Japan}
\email{yuki.mifune.c9@math.nagoya-u.ac.jp}
\urladdr{\url{https://orcid.org/0009-0004-0567-4132}}

\author[Yuya Otake]{Yuya Otake}
\address{Graduate School of Mathematics, Nagoya University, Furocho, Chikusaku, Nagoya 464-8602, Japan}
\email{m21012v@math.nagoya-u.ac.jp}
\urladdr{\url{https://orcid.org/0009-0000-2866-6902}}


\DeclareMathOperator{\depth}{depth}

\DeclareMathOperator{\h}{H}

\newcommand{\Z}{\mathbb{Z}}

\newcommand{\A}{\mathcal{A}}

\newcommand{\T}{\mathcal{T}}

\newcommand{\D}{\mathsf{D}}

\newcommand{\N}{\mathbb{N}}

\pgfdeclarelayer{bg}    
\pgfsetlayers{bg,main} 
\newcommand{\x}{{\bm{x}}}

\newcommand{\m}{\mathfrak{m}}

\newcommand{\p}{\mathfrak{p}}

\DeclareMathOperator{\pd}{pd}

\DeclareMathOperator{\add}{add}

\DeclareMathOperator{\ca}{ca}
\DeclareMathOperator{\coca}{coca}
\DeclareMathOperator{\CA}{\mathfrak{ca}}

\DeclareMathOperator{\gldim}{gl.dim}
\DeclareMathOperator{\res}{res}

\DeclareMathOperator{\id}{id}
\DeclareMathOperator{\Spec}{Spec}

\DeclareMathOperator{\ann}{ann}

\DeclareMathOperator{\Hom}{Hom}

\DeclareMathOperator{\ul}{Ul}
\DeclareMathOperator{\Ext}{Ext}

\DeclareMathOperator{\End}{End}

\DeclareMathOperator{\Max}{Max}

\def\Rfd{\operatorname{Rfd}}
\DeclareMathOperator{\ara}{ara}

\DeclareMathOperator{\thick}{\mathsf{thick}}

\DeclareMathOperator{\Reg}{Reg}
\DeclareMathOperator{\Sing}{Sing}
\DeclareMathOperator{\IPD}{IPD}
\DeclareMathOperator{\IID}{IID}
\DeclareMathOperator{\Supp}{Supp}

\DeclareMathOperator{\mo}{mod}
\DeclareMathOperator{\CM}{CM}
\DeclareMathOperator{\OCM}{\Omega CM}
\DeclareMathOperator{\syz}{\Omega}

\newcommand{\sg}{\mathsf{sg}}

\newtheorem*{theorem*}{Theorem}
\newtheorem{theorem}{Theorem}[section]

\newtheorem*{Question}{Question}

\newtheorem{proposition}[theorem]{Proposition}

\newtheorem{lemma}[theorem]{Lemma}
\newtheorem{corollary}[theorem]{Corollary}

\theoremstyle{definition}
\newtheorem{example}[theorem]{Example}
\newtheorem{remark}[theorem]{Remark}

\newtheorem{chunk}[theorem]{}

\usepackage[utf8]{inputenc}

\newtheorem*{ack}{Acknowledgments}

\title[Generation of singularity categories and infinite injective dimension locus]{Generation of singularity categories and infinite injective dimension locus via annihilation of cohomologies}

\begin{document}
\maketitle
\begin{abstract}
Let $R$ be a commutative Noetherian ring. We establish a close relationship between the strong generation of the singularity category of $R$ and the nonvanishing of the annihilator of the singularity category of $R$. As an application,  we prove that the singularity category of $R$ has a strong generator if and only if the annihilator of the singularity category of $R$ is nonzero when $R$ is a Noetherian domain with Krull dimension at most one. 
We introduce the notion of the co-cohomological annihilator of modules. If the category of finitely generated $R$-modules has a strong generator, we show that the infinite injective dimension locus of a finitely generated $R$-module $M$ is closed, with the defining ideal given by the co-cohomological annihilator of $M$. Finally, we provide a connection between the existence of an extension generator of the category of finitely generated $R$-modules and the finiteness of the Krull dimension of $R$. 
\end{abstract}
\section{Introduction}

Let $R$ be a commutative Noetherian ring. In \cite{IT2016}, 
Iyengar and Takahashi introduced the notion of the cohomological annihilator of a ring $R$, denoted $\CA(R)$; cf. \ref{def of ca(R)}. They observed a close relationship between the condition $\CA(R)\neq 0$ and the strong generation of the category $\mo(R)$ of finitely generated $R$-modules; cf. \ref{def of gen for Abel}. 
It is established in \cite{IT2016}  that $\mo(R)$ has a strong generator when $R$ is a localization of a finitely generated algebra over a field or an equicharacteristic excellent local ring. 

For a commutative Noetherian domain $R$, Elagin and Lunts \cite{Elagin-Lunts:2018} observed that $\CA(R)\neq 0$ if the bounded derived category of $R$, denoted $\D^f(R)$, has a strong generator; cf. \ref{def of gen for tri}. Recently, the first author, Lank, and Takahashi \cite{DLT} established that the condition $\CA(R/\p)\neq 0$ for each prime ideal $\p$ of $R$ is equivalent to that $\mo(R/\p)$ has a strong generator for each prime ideal $\p$ of $R$, which in turn is equivalent to that  $\D^f(R/\p)$ has a strong generator for each prime ideal $\p$ of $R$. If these conditions hold, they show that $\mo(R)$ has a strong generator. As an application, it is proved in \cite{DLT} that $\mo(R)$ has a strong generator for any quasi-excellent ring $R$ with finite Krull dimension.

The \emph{singularity category} of $R$, denoted $\D_{\sg}(R)$, was introduced by Buchweitz \cite{Buchweitz} and Orlov \cite{Orlov} as the Verdier quotient of the bounded derived category by the full subcategory of perfect complexes. We are motivated by the natural question: 
\begin{Question}
    How can the (strong) generation of the singularity category be characterized?
\end{Question}

 We investigate the above question in Section \ref{Section singularity cat}. It turns out that there is a close connection between the strong generation of the singularity category of $R$ and the nonvanishing of the annihilator of the singularity category of $R$, as well as the nonvanishing of the cohomological annihilator of modules.
 
The \emph{annihilator of the singularity category} of $R$, denoted $\ann_R\D_{\sg}(R)$, is an ideal of $R$ consisting of elements in $R$ that annihilate all homomorphisms of complexes in $\D_{\sg}(R)$; cf. \ref{def of ann of sin cat}. The ideal $\ann_R\D_{\sg}(R)$ measures the singularity of $R$ in the sense that $R$ is regular if and only if $\ann_R\D_{\sg}(R)=R$. The annihilator of the singularity category has recently attracted increasing interest and has been studied in \cite{Esentepe,  Liu, Mifune}.  
For a commutative Noetherian domain $R$, if the singularity category of $R$ has a strong generator, then the annihilator $\ann_R\D_{\sg}(R)$ is nonzero; see Corollary \ref{annihilator of sin}. It is natural to ask whether the converse holds in general. We prove that if, in addition, the Krull dimension of $R$ is less than or equal to one, the converse holds. This is an immediate consequence of \Cref{T1} (2).
\begin{theorem}\label{T1}(See \ref{main},  \ref{main-second}, and \ref{equi-resolving})
    Let $R$ be a commutative Noetherian ring. Then:
\begin{enumerate}
    \item 
    The following two conditions are equivalent.
    \begin{enumerate}
        \item $\D_{\sg}(R/\p)$ has a generator for each prime ideal $\p$ of $R$.

        \item $\displaystyle\bigcap_{M\in\mo(R/\p)}\sqrt{\ca_{R/\p}(M)}\neq 0$ for each prime ideal $\p$ of $R$.
    \end{enumerate} 
Moreover, if $R$ has finite Krull dimension, the above are equivalent to $\mo(R/\p)$ having an extension generator (cf. \ref{def of ext gen}) for each prime ideal $\p$ of $R$.

    \item  Assume $R$ is non-regular with isolated singularities. Then the following are equivalent.

\begin{enumerate}
\item
$\D_{\sg}(R)$ has a strong generator.

\item $R/\ann_{R}\D_{\sg}(R)$ is Artinian.

\item $R/(\displaystyle\bigcap_{M\in\mo(R)}\ca_R(M))$ is  Artinian. 

\item
$\mo(R)$ has a point-wise strong generator; {\rm cf.} \ref{up to syzygy}.

\end{enumerate}
\end{enumerate}
 
\end{theorem}

Theorem \ref{T1} is the main result in Section \ref{Section singularity cat} regarding the characterization of the (strong) generation of singularity categories. The key input is the notion of the \emph{cohomological annihilator} of a finitely generated $R$-module $M$, denoted $\ca_R(M)$, in Section \ref{Section ca of modules}.
The key ingredient in proof of Theorem \ref{T1} is the observation that $\ann_R\D_{\sg}(R)$ coincides with $\displaystyle\bigcap_{M\in\mo(R)}\ca_R(M)$; see Corollary \ref{annihilator of sin}. 
When $R$ is, in addition, a Gorenstein local ring,  Theorem \ref{T1} (2) was proved by Bahlekeh, Hakimian, Salarian, and Takahashi \cite{BHST} in terms of the cohomological annihilator $\CA(R)$; see Remark \ref{compare with BHST}. 

For a finitely generated $R$-module $M$, it is known that the finite projective dimension locus of $M$, 
$
\{\p\in\Spec(R)\mid \pd_{R_\p}(M_\p)<\infty\},
$
is an open subset in the Zariski spectrum $\Spec(R)$; see \cite{BM}. However, the same is not true for finite injective dimension locus, even when the module is the ring itself, i.e., even the Gorenstein locus can fail to be open; see \cite{Nishimura}. 
For an excellent ring $R$, Greco and Marinari \cite{Greco-Marinari} observed that the Gorenstein locus is open. Takahashi \cite{Takahashi:2006_Glasgow} extended this result by proving that the finite injective dimension locus of a finitely generated module over an excellent ring is open. 

In Section \ref{Section coca of modules}, we investigate the \emph{infinite injective dimension locus} $\IID(M)$ of an $R$-module $M$, which is the complement of the finite injective dimension locus. Namely, $$\IID(M)=\{\p\in \Spec(R)\mid \id_{R_\p}(M_\p)=\infty\}.$$ 
We observe in  \Cref{T2} that the infinite injective dimension locus of a finitely generated module is always closed if $\mo(R)$ has an \emph{extension generator} (cf. \ref{def of ext gen}).  
 For a quasi-excellent ring $R$ with finite Krull dimension and a finitely generated $R$-module $M$,  it follows from \cite{DLT} that $\mo(R)$ has a strong generator. Consequently, by Theorem \ref{T2}, $\IID(M)$ is closed. In Section \ref{Section coca of modules}, we introduce the notion of the \emph{co-cohomological annihilator} of an $R$-module $M$, denoted $\coca_R(M)$. The following result shows that, under some mild assumptions, this ideal is a defining ideal of the infinite injective dimension locus.

\begin{theorem}\label{T2} (See \ref{singid} and \ref{equality})
      Let $R$ be a commutative Noetherian ring and $M$ be a finitely generated $R$-module. Then:

\begin{enumerate}
    \item  If $\mo(R)$ has an extension generator, then $\IID(M)$ is closed in $\Spec(R)$. Moreover, $\Sing(R)=\IID(G)$ holds for each extension generator $G$ of $\mo(R)$.

    \item If $\mo(R)$ has a strong generator, then $$\IID(M)=V(\coca_R(M)).$$
\end{enumerate}

\end{theorem}


 We also compare the cohomological annihilator with the co-cohomological annihilator of a finitely generated module $M$. It turns out that these annihilators coincide when $R$ is Gorenstein with finite Krull dimension and $M$ is maximal Cohen--Macaulay; see Lemma \ref{example of co ca}. 



Our third result, Theorem \ref{new theorem}, establishes a connection between the existence of an extension generator of the module category and the finiteness of the Krull dimension. Moreover, Theorem \ref{new theorem} (2) extends a recent result of Araya, Iima, and Takahashi \cite{AIT} on the generation of syzygy modules out of a single module by only taking direct summands and extensions; see Remark \ref{AIT}.

\begin{theorem}\label{new theorem} (See \ref{finite-extension gen} and \ref{ext-dimension})
    Let $R$ be a commutative Noetherian ring. Then:

\begin{enumerate}
    \item If $\mo(R)$ has an extension generator, then the Krull dimension of $R$ is finite.

\item If the Krull dimension of $R$ is finite and the singular locus of $R$ is a finite set, then $\mo(R)$ has an extension generator. 




\end{enumerate}
\end{theorem}


\begin{ack}
We would like to thank Kaito Kimura, Jan Šťovíček, and Ryo Takahashi for their helpful discussions and valuable comments related to this work. The first author was partially supported by the Charles University Research Center program No.UNCE/SCI/022 and a grant GA \v{C}R 23-05148S from the Czech Science Foundation. The second author was supported by the National Natural Science Foundation of China (No. 12401046) and the Fundamental Research Funds for the Central Universities (Nos. CCNU25JC025, CCNU25JCPT031). The fourth author was partly supported by Grant-in-Aid for JSPS Fellows 23KJ1119.
\end{ack}

\section{Notation and Terminology}
Throughout this article, $R$ will be a commutative Noetherian ring. For each $R$-module $M$, let $\pd_R(M)$ denote the projective dimension of $M$ over $R$, and $\id_R(M)$ denote the injective dimension of $M$ over $R$.  We write $\mo(R)$ to be the category of finitely generated $R$-modules. Let $\dim(R)$ denote the Krull dimension of $R$, and $\gldim(R)$ denote the global homological dimension of $R$. 

\begin{chunk}\label{def of Gorenstein}
    \textbf{(Strongly) Gorenstein rings.} A commutative Noetherian ring $R$ is called \emph{strongly Gorenstein} provided that $\id_R(R)<\infty$. For instance, regular local rings and complete intersection rings are both examples of strongly Gorenstein rings.

 A commutative Noetherian ring $R$ is said to be \emph{Gorenstein} provided that $R_\p$ is strongly Gorenstein (i.e., $\id_{R_\p}(R_\p)<\infty$) for each prime ideal $\p$ of $R$. Note that a commutative Noetherian local ring is Gorenstein if it is strongly Gorenstein.
 
 For a Gorenstein ring $R$, $\id_R(R)=\dim(R)$; this can be proved by combining \cite[Theorem 3.1.17]{BH} and \cite[Corollary 2.3]{Bass1963}.
In particular, a commutative Noetherian ring $R$ is strongly Gorenstein if and only if it is Gorenstein with finite Krull dimension. 
\end{chunk}

\begin{chunk}\label{def regular}
\textbf{Regular rings.}
    A commutative Noetherian local ring is called \emph{regular} if its maximal ideal can be generated by a system of parameter. Auslander, Buchsbaum, and Serre observed that a commutative Noetherian local ring is regular if and only if its global homological dimension is finite; see \cite[Theorem 2.2.7]{BH}. 

    A commutative Noetherian ring $R$ is called regular if $R_\p$ is regular for each prime ideal $\p$ of $R$. Furthermore, a commutative Noetherian ring is regular if and only if every finitely generated module has finite projective dimension; see \cite[Lemma 4.5]{BM} for the forward direction, and the backward direction follows from the criterion of Auslander, Buchsbaum, and Serre.

    For a regular ring, $\gldim(R)=\dim(R)$; see \cite[Theorem 5.94]{Lam}. In particular, a commutative Noetherian ring is regular with finite Krull dimension if and only if $\gldim(R)$ is finite.
\end{chunk}
\begin{chunk}\label{def ann and supp}
\rm  \textbf{Annihilator and support of modules.} 
For each $R$-module $M$, let $\ann_R(M)$ denote the \emph{annihilator} of $M$ over $R$. That is, $\ann_R(M)\colonequals\{r\in R\mid r\cdot M=0\}$. 

The set of all prime ideals of $R$ is denoted by $\Spec(R)$. It is endowed with the Zariski topology; the closed subset in this topology is of the form $V(I)\colonequals\{\p\in \Spec(R)\mid \p\supseteq I\}$ for each ideal $I$ of $R$. 

For each $R$-module $M$, the \emph{support} of $M$ is 
$$
\Supp_RM\colonequals\{\p\in \Spec(R)\mid M_\p\neq 0\},
$$
where $M_\p$ is the localization of $M$ at $\p$. Note that $\Supp_R(M)\subseteq V(\ann_R(M))$; the equality holds if, in addition, $M$ is finitely generated. 
\end{chunk}

\begin{chunk}\label{def of syzygy}
 \textbf{Syzygy modules.}   For a finitely generated $R$-module $M$ and $n\geq 1$, we let $\Omega^n_R(M)$ denote the $n$-th syzygy of $M$. That is, there is a long exact sequence
$$
0\rightarrow \Omega^n_R(M)\rightarrow P^{-(n-1)}\rightarrow \cdots P^{-1}\rightarrow P^0\rightarrow M\rightarrow 0,
$$
where $P^{-i}$ are finitely generated projective $R$-modules for $0\leq i\leq n-1$. By Schanuel's Lemma, $\Omega^n_R(M)$ is independent of the choice of the projective resolution of $M$ up to projective summands. By convention, $\Omega^0_R(M)=M$ for each finitely generated $R$-module $M$. Let $\mathcal C$ be a full subcategory of $\mo(R)$, we denote
$$
\Omega^n_R(\mathcal C)\colonequals\{\Omega^n_R(M)\mid M\in \mathcal{C}\}.
$$

We say a finitely generated $R$-module $M$ is an \emph{infinite syzygy} if there exists an exact sequence
$
0\rightarrow M\rightarrow Q^0\rightarrow Q^1\rightarrow Q^2\rightarrow \cdots,
$
where $Q^i$ are finitely generated projective $R$-modules for $i\geq 0$. 
\end{chunk}
\begin{chunk}\label{def of MCM}
    \textbf{Maximal Cohen--Macaulay modules.} For a finitely generated $R$-module $M$, it is said to be \emph{maximal Cohen--Macaulay} provided that $\depth(M_\p)\geq \dim(R_\p)$ for each prime ideal $\p$ of $R$, where $\depth(M_\p)$ represents the depth of $M_\p$ over $R_\p$; see details in \cite{BH}. Let $\CM(R)$ denote the full subcategory of $\mo(R)$ consisting of maximal Cohen--Macaulay $R$-modules. A commutative Noetherian ring is \emph{Cohen--Macaulay} if $R$ is in $\CM(R)$. 

    Assume $R$ is a Gorenstein ring. Note that a finitely generated $R$-module $M$ is maximal Cohen--Macaulay if and only if $\Ext^{i}_R(M,R)=0$ for all $i>0$; this can be proved by using Ischebeck's formula (see \cite[Exercise 3.1.24]{BH}). Then it follows from \cite[Theorem B.1.6]{Buchweitz} that, for each $M\in \mo(R)$, $\Omega^s_R(M)$ is maximal Cohen--Macaulay for some $s\geq 0$. On the other hand, any maximal Cohen--Macaulay $R$-module is an infinite syzygy; see \cite[Theorem B.1.3]{Buchweitz}. 
\end{chunk}
\begin{chunk}
      \textbf{Thick subcategories of Abelian categories.}
     Let $\A$ be an Abelian category.  A full subcategory $\mathcal C$ of $\A$ is called \emph{thick} if $\mathcal C$ is closed under direct summands, and it contains objects that fit into a short exact sequence 
     such that the other two objects are in $\mathcal C$. 

     For each object $X$ in $\A$. Let $\thick_\A(X)$ denote the smallest thick subcategory of $\A$ containing $X$; the notation $\thick_\A$ will have no confusion with $\thick_\T$ in \ref{def of thick over tri} for a triangulated category $\T$. This can be constructed inductively; it is an analogous construction of a thick subcategory in a triangulated category (see \ref{def of thick over tri}).

     Set $\thick_\A^0(X)\colonequals \{0\}$. Let $\thick_\A^1(X)$ be the full subcategory of $\A$ consisting of objects which are direct summands of any finite copies of $X$. For $n\geq 2$, let $\thick_\A^n(X)$ denote the full subcategory of $\A$ consisting of direct summands of any object in $\A$ that fits into a short exact sequence 
     $$
0\rightarrow Y_1\rightarrow Y_2\rightarrow Y_3\rightarrow 0,
$$
where the other two objects satisfying: one is in $\thick^{n-1}_\A(X)$, and the other one is in $\thick^{1}_\A(X)$. Note that $\thick_\A(X)=\displaystyle\bigcup_{n\geq 0}\thick_\A^n(X).$
     \end{chunk}
\begin{chunk}\label{def of thick over tri}
    \textbf{Thick subcategories of triangulated categories.} Let $\T$ be a triangulated category with a suspension functor $[1]$. A full category $\mathcal C$ of $\T$ is called \emph{thick} if it is closed under suspensions, cones, and direct summands. 

    For each object $X$ in $\T$, let $\thick_\T(X)$ denote the smallest thick subcategory of $\T$ containing $X$. $\thick_\T(X)$ can be constructed inductively as below; see \cite[Section 2]{ABIM} for more details.

    First, $\thick_{\T}^0(X)\colonequals \{0\}$. Let $\thick_\T^1(X)$ be the smallest full subcategory of $\T$ containing $X$, and it is closed under suspensions,  finite direct sums, and direct summands. Inductively,  $\thick_\T^n(X)$ is denoted to be the full subcategory of $\T$ consisting of objects $Y$ in $\T$ that appear in an exact triangle
$$
Y_1\rightarrow Y\oplus Y^\prime\rightarrow Y_2\rightarrow Y_1[1],
$$
where $Y_1\in \thick^{n-1}_{\T}(X)$ and $Y_2\in \thick_{\T}^{1}(X)$.
Note that $\thick_\T(X)=\displaystyle\bigcup\limits_{n\geq 0}\thick_\T^n(X).$
\end{chunk}

\begin{chunk}\label{def of sin cat}
\rm \textbf{Derived categories and singularity categories.}
Let $\D(R)$ denote the derived category of $R$-modules. This is a triangulated category with a suspension functor $[1]$; for each complex $X$, $(X[1])^i=X^{i+1}$ and $\partial_{X[1]}=-\partial_X$.
The \emph{bounded derived category}, denoted $\D^f(R)$, is the full subcategory of $\D(R)$ consisting of  complexes $X$ such that its total cohomology, denoted $\displaystyle\bigoplus_{i\in \Z}\h^i(X)$, is finitely generated over $R$; note that $\D^f(R)$ is a thick subcategory of $\D(R)$.

For each complex  $X$ in $\D(R)$, it is called \emph{perfect} if $X\in \thick_{\D(R)}(R)$. It turns out that a complex in $\D(R)$ is perfect if and only if it is isomorphic to a bounded complex of finitely generated projective $R$-modules; this can be proved by using \cite[Lemma 1.2.1]{Buchweitz}. Note that $\thick_{\D(R)}(R)\subseteq \D^f(R)$.

The \emph{singularity category} of $R$ is defined to be the Verdier quotient  $$\D_{\sg}(R)\colonequals\D^f(R)/\thick_{\D(R)}(R).$$
This category was introduced by Buchweitz \cite[Definition 1.2.2]{Buchweitz} under the name ``stable derived category" and later also found by Orlov \cite[1.2]{Orlov}. This category detects the singularity of $R$ in the sense that $\D_{\sg}(R)$ is trivial if and only if $R$ is regular.
\end{chunk}
   The following definition of ``strong generator" in \ref{def of gen for Abel} differs from that of Iyengar and Takahashi \cite[4.3]{IT2016}. However, for a commutative Noetherian ring $R$, $\mo(R)$ has a strong generator in the sense of \ref{def of gen for Abel} if and only if $\mo(R)$ has a strong generator in the sense of Iyengar and Takahashi; see \cite[Corollary 4.6]{IT2016}.
\begin{chunk}\label{def of gen for Abel}
    \textbf{(Strong) generators of Abelian categories.}
    Let $\A$ be an Abelian category and $G$ be an object in $\A$. The object $G$ is called a \emph{generator} of $\A$ if $\thick_\A(G)=\A$, and $G$ is called a \emph{strong generator} of $\A$ if $\thick^n_\A(G)=\A$ for some $n\geq 0$. 
    
For example, if $R$ is an  Artinian ring $R$, $R/J(R)$ is a strong generator of $\mo(R)$, where $J(R)$ is the Jacobson radical of $R$. Moreover, $\mo(R)=\thick^{\ell\ell(R)}_{\mo(R)}(R/J(R))$, where $\ell\ell(R)\colonequals {\rm inf}\{n\geq 0\mid J(R)^n=0\}$ is the Loewy length of $R$.
\end{chunk}
\begin{chunk}\label{def of gen for tri}
    \textbf{(Strong) generators of triangulated categories.}
    Let $\T$ be a triangulated category and $G$ be an object in $\T$. The object $G$ is called a \emph{generator} of $\T$ if $\thick_\T(G)=\T$, and $G$ is called a \emph{strong generator} of $\T$ if $\thick^n_\T(G)=\T$ for some $n\geq 0$; see details in \cite{Bondal/vdB:2003} and \cite{Rouquier}. Note that $\T$ has a strong generator if and only if the Rouquier dimension (see \cite{Rouquier}) of $\T$ is finite.

If $R$ is an Artinian ring, then $R/J(R)$ is a strong generator of $\D^f(R)$; indeed, $\D^f(R)=\thick^{\ell\ell(R)}_{\D(R)}(R/J(R))$ by \cite[Proposition 7.37]{Rouquier}.
\end{chunk}
We end this section by recording the (strong) generators of the regular rings.
\begin{chunk}\label{gen regular}
Let $R$ be a commutative Noetherian ring. Note that, for each $G\in \mo(R)$, $G$ is a generator of $\D^f(R)$ if $G$ is a generator of $\mo(R)$. Combining with this, the following are equivalent.
  \begin{enumerate}
    \item $R$ is regular.

    \item $R$ is a generator of $\mo(R)$.

    \item $R$ is a generator of $\D^f(R)$. 
\end{enumerate}
\end{chunk}
  
\begin{chunk}\label{strong gen regular finite dim}
Let $R$ be a commutative Noetherian ring and $G$ be a module in $\mo(R)$.  
   If $\mo(R)=\thick_{\mo(R)}^n(G)$, then $\D^f(R)=\thick_{\D(R)}^{2n}(G)$ by \cite[Lemma 7.1]{IT2016}. In particular, $G$ is a strong generator of $\D^f(R)$ if $G$ is a strong generator of $\mo(R)$. This yields $(2)\Rightarrow (3)$ below. Indeed, all the following conditions are equivalent.
  \begin{enumerate}
    \item $R$ is regular with $\dim(R)<\infty$, equivalently $\gldim(R)<\infty$ (see \ref{def regular}).

    \item $R$ is a strong generator of $\mo(R)$.

    \item $R$ is a strong generator of $\D^f(R)$. 
\end{enumerate}
$(1)\Rightarrow (2)$: if $\gldim(R)=d$ is finite, then $\mo(R)=\thick_{\mo(R)}^{d+1}(R)$. $(3)\Rightarrow (1)$: if $\D^f(R)=\thick^n_{\D(R)}(R)$, then \cite[Proposition 4.5 and 4.6]{Christensen} yields that $\pd_R(M)\leq n-1$ for each $M\in \mo(R)$, and hence $\gldim(R)\leq n-1$. 

Moreover, if the above conditions hold, then $\mo(R)=\thick_{\mo(R)}^{d+1}(R)$ and $\D^f(R)=\thick_{\D(R)}^{d+1}(R)$ (see \cite[Theorem 8.3]{Christensen} ), where $d=\dim(R)$. 
\end{chunk}

\section{Cohomological annihilators for modules}\label{Section ca of modules}
The following definition of cohomological annihilators for modules is inspired by the definition of cohomological annihilators for rings introduced by Iyengar and Takahashi \cite{IT2016}; see their definition in \ref{def of ca(R)}. For a subset $I\subseteq \N$ and full subcategories $\mathcal{C},\mathcal{D}$ of $\mo(R)$, set $\Ext^I_R(\mathcal{C},\mathcal{D})\colonequals \bigoplus\limits_{i\in I}\bigoplus\limits_{M\in\mathcal{C},N\in\mathcal{D}}\Ext^i_R(M,N)$. 
\begin{chunk}\label{def ca of modules}
    For each finitely generated $R$-module $M$ and $n\geq 0$, we define the $n$-th \emph{cohomological annihilator} of $M$ to be
    $$
    \ca_R^n(M)\colonequals \ann_R\Ext_R^{\geq n}(M,\mo(R)).
    $$ 
That is, $\ca^n_R(M)$ consists of elements $r\in R$ such that $r\cdot \Ext^i_R(M,N)=0$ for each $i\geq n$ and $N\in \mo(R)$. Consider the ascending chain of ideals
$$
\ann_R(M)=\ca_R^0(M)\subseteq \ca_R^1(M)\subseteq\ca_R^2(M)\subseteq\cdots,
$$
 the cohomological annihilator of $M$ is defined to be the union of these ideals
$$
\ca_R(M)\colonequals \bigcup_{n\geq 0}\ca_R^n(M).
$$
Since $R$ is Noetherian,  $\ca_R(M)=\ca_R^n(M)$ for $n\gg 0$.

Note that $\ca_R^n(M)=R$ if and only if $\pd_R(M)<n$. Thus, $\ca_R(M)=R$ if and only if $\pd_R(M)<\infty$.
\end{chunk}

\begin{chunk}
    Let $\mathcal C$ be a full subcategory of $\mo(R)$, Bahlekeh, Salarian, Takahashi, and Toosi \cite[Definition 5.7]{BSTT2022} introduced the concept of \emph{cohomological annihilator} of $\mathcal C$, denoted $\ca_R(\mathcal C)$.  It is defined as
    $$
\ca_R(\mathcal C)\colonequals \bigcap_{X\in \mathcal C}\ca_R(X). 
$$

By definition, if $\mathcal C$ consists of only one module $X$, then $\ca_R(\mathcal C)=\ca_R(X)$.
\end{chunk}

\begin{chunk}\label{def of ca(R)}
For each $n\geq 0$, following \cite[Definition 2.1]{IT2016}, the $n$-th \emph{cohomological annihilator} of the ring $R$ is defined to be
$$
\CA^n(R)\colonequals \ann_R\Ext_R^{\geq n}(\mo(R),\mo(R)).
$$
By definition, $\CA^n(R)=\displaystyle\bigcap\limits_{M\in \mo(R)}\ca_R^n(M)$ and there is an ascending chain of ideals
$$
0=\CA^0(R)\subseteq \CA^1(R)\subseteq \CA^2(R)\subseteq\cdots.
$$
The cohomological annihilator of the ring $R$ is defined to be the ideal
$$\CA(R)\colonequals \bigcup_{n\geq 0}\CA^n(R).$$
Also, $\CA(R)=\CA^n(R)$ for $n\gg 0$ as $R$ is Noetherian. 
\end{chunk}

\begin{example}
    Let $R$ be an equicharacteristic complete Cohen--Macaulay local ring with Krull dimension $d$. Wang \cite[Theorem 5.3]{Wang1994} observed that $\CA^{d+1}(R)$ contains the Jacobian ideal of $R$; see also \cite[Theorem 1.2]{IT2021}.
\end{example}

\begin{lemma}\label{test regular}
(1) $\CA(R)\subseteq \ca_R(\mo(R))$.

(2) $\ca_R(\mo(R))=R$ if and only if $R$ is regular.

(3) $\CA(R)=R$ if and only if $R$ is regular with finite Krull dimension.
\end{lemma}
\begin{proof}
  (1) For each $n\geq 0$, we have
    $$
   \CA^n(R)=\bigcap_{M\in \mo(R)}\ca_R^n(M)\subseteq \bigcap_{M\in \mo(R)}\ca_R(M)=\ca_R(\mo(R)).
    $$
 Thus, $\CA(R)\subseteq \ca_R(\mo(R))$ as $\CA(R)=\CA^n(R)$ for $n\gg 0$.

    (2) By definition, $\ca_R(\mo(R))=R$ if and only if $\ca_R(M)=R$ for each $M\in \mo(R)$. This is equivalent to $\pd_R(M)<\infty$ for each $M\in \mo(R)$; see \ref{def ca of modules}. This is equivalent to that $R$ is regular; see \ref{def regular}. 

    (3) By \cite[Example 2.5]{IT2016}, $\CA(R)=R$ if and only if $\gldim(R)<\infty$.  The desired result now follows immediately from \ref{def regular}.
\end{proof}
In general, $\CA(R)\neq \ca_R(\mo(R))$; see the example below. They are equal if, in addition, $R$ is strongly Gorenstein; see Proposition \ref{relation CA(R) and ca(R-mo)}.
\begin{example}\label{ca proper}
    The inclusion in Lemma \ref{test regular} (1) can be proper. In \cite[Appendix, Example 1]{Nagata}, Nagata constructed a commutative Noetherian regular ring $R$ with infinite Krull dimension. In this case, $\CA(R) \subsetneqq \ca_R(\mo(R))=R$ by Lemma \ref{test regular}.
\end{example}

\begin{chunk}\label{localization}
    Let $M,N$ be $R$-modules. If $M$ is finitely generated, then it follows from \cite[Proposition 3.3.10]{Weibel} that there is a natural isomorphism
    $$
\Ext^n_R(M,N)_\p\cong \Ext^n_{R_\p}(M_\p,N_\p)
$$
for each $\p\in \Spec(R)$ and $n\geq 0$. 
\end{chunk}
\begin{lemma}\label{basic}
      Let $M$ be a finitely generated $R$-module and $n\geq 1$. Then: 
\begin{enumerate}
    \item $\ca_R^n(M)=\ann_R\Ext^n_R(M,\Omega^n_R(M))=\ann_R\underline{\End}_R(\Omega^{n-1}_R(M))$. 

    \item  $\{\p\in \Spec(R)\mid \pd_{R_\p}(M_\p)\geq n\}=V(\ca_R^n(M))=\Supp_R\Ext_R^n(M,\Omega^n_R(M))
$.

    \item $\ca_R^n(M)=\ann_R\Ext^n_R(M,\mo(R))$.

    \item $\ca_R^{n+i}(M)=\ca_R^n(\Omega^i_R(M))$ for each $i\geq 1$. Hence, $\ca_R(M)=\ca_R(\Omega^i_R(M))$ for each $i\geq 1$.

    \item Let $S$ be a multiplicatively closed subset of $R$. Then $S^{-1}\ca^n_R(M)=\ca^n_{S^{-1}R}(S^{-1}M)$. In particular, 
    $S^{-1}\ca_R(M)=\ca_{S^{-1}R}(S^{-1}M)$. 
\end{enumerate}
\end{lemma}
\begin{proof}
    (1) This follows from \cite[Lemma 2.14]{IT2016} and \cite[Lemma 3.8]{DeyTa}.

    (2) Combining with \ref{localization}, $\pd_{R_\p}(M_\p)\geq n$ if and only if $\p\in \Supp_R\Ext^n_R(M,\Omega^n_R(M))$. 
 Since $\Ext_R^n(M,\Omega^n_R(M))$ is finitely generated over $R$, 
    $$ \Supp_R\Ext^n_R(M,\Omega^n_R(M))=V(\ann_R\Ext^n_R(M,\Omega^n_R(M)).
    $$
Combining this with (1), the desired statement follows.

    (3) It is clear that 
    $$
\ca_R^n(M)\subseteq \ann_R\Ext^n_R(M,\mo(R))\subseteq \ann_R\Ext^n_R(M,\Omega^n_R(M)).
$$
By (1), these inclusions are equal.

    (4) For each $X\in \mo(R)$, consider the isomorphism
    $$
    \Ext^n_R(\Omega_R^1(M),X)\cong \Ext^{n+1}_R(M,X).
    $$
Combining with (3), we have $\ca_R^{n+1}(M)=\ca_R^n(\Omega^1_R(M))$. Then the first statement of (4) can be obtained by induction on the number $i$. Thus, for each $i\geq 1$,
$$
\ca_R(M)=\bigcup_{n\geq 0}\ca_R^{n+i}(M)=\bigcup_{n\geq 0}\ca_R^n(\Omega^i_R(M))=\ca_R(\Omega^i_R(M)).
$$
This completes the proof.

(5) The second statement follows from the first one. For the first statement, combining with (4), it suffices to prove the case of $n=1$. By (1), we get  an inclusion $S^{-1}\ca^1_R(M)\subseteq\ca^1_{S^{-1}R}(S^{-1}M)$. For the converse, assume $r/s\in \ca^1_{S^{-1}R}(S^{-1}M)$, where $s\in S$ and $r\in R$. It remains to prove that $r/s\in S^{-1}\ca^1_R(M)$. By (1), we get that the multiplication $r/s\colon S^{-1}M\rightarrow S^{-1}M$ factors through a finitely generated projective $S^{-1}R$-module, and hence $r/s\colon S^{-1}M\rightarrow S^{-1}M$ factors through a finitely generated free $S^{-1}R$-module. Combining with the fact that the isomorphism $S^{-1}\Hom_R(X,Y)\cong \Hom_{S^{-1}R}(S^{-1}X,S^{-1}Y)$ holds for $X,Y\in\mo(R)$, we conclude that there is a factorization
$$
\xymatrix{
S^{-1}M\ar[rr]^-{r/s}\ar[rd]_-{\alpha/s_1}&& S^{-1}M\\
& S^{-1}F\ar[ru]_-{\beta/s_2}& 
},
$$
where $F$ is a finitely generated free $R$-module, $s_1,s_2\in S$, $\alpha\colon M\rightarrow F$, and $\beta\colon F\rightarrow M$. From the factorization, it follows that there exists $t\in S$ such that $ts_2s_1r=ts\beta\alpha$. Note that $r/s=(ts_2s_1r)/(ts_2s_1s)$ and $ts_2s_1r=ts\beta\alpha$ factors through $F$. By (1), $ts_2s_1r\in \ca^1_R(M)$, and hence $r/s\in S^{-1}\ca_R^1(M)$. 
\end{proof}


\begin{chunk}\label{finite}
  Let $M$ be a finitely generated $R$-module.
  
(1) Following \cite[Definition 2.9]{Takahashi2014}, the \emph{infinite projective dimension locus} of $M$ is defined to be
$$
\IPD(M)\colonequals \{\p\in \Spec(R)\mid \pd_{R_\p}(M_\p)=\infty\}.
$$ 

(2) $\sup\{\pd_{R_\p}(M_\p)\mid \p\notin \IPD(M)\}$ is finite. Indeed, it follows from the proof of \cite[Lemma 4.5]{BM} that $$\{\p\in \Spec(R)\mid \pd_{R_\p}(M_\p)<\infty\}=\{\p\in \Spec(R)\mid \pd_{R_\p}(M_\p)\leq n\}$$
for some $n\geq 0$, and hence $\sup\{\pd_{R_\p}(M_\p)\mid \p\notin \IPD(M)\}\leq n$; see also \cite[Theorem 1.1]{AIL}.
\end{chunk}
\begin{proposition}\label{relation}
    Let $M$ be a finitely generated $R$-module. Then
$$
\IPD(M)=V(\ca_R(M))= V(\ca_R^{d+1}(M)),
$$
where $d=\sup\{\pd_{R_\p}(M_\p)\mid \p\notin \IPD(M)\}$.
\end{proposition}
\begin{proof}
     The first equality follows from the following:
    \begin{align*}
        \IPD(M)& =\bigcap_{n\geq 1}\Supp_R \Ext^n_R(M,\Omega^n_R(M))\\
        & =\bigcap_{n\geq 1}V(\ca_R^n(M))\\
        &= V(\bigcup_{n\geq 1}\ca_R^n(M))\\
        &=V(\ca_R(M)),
    \end{align*}
  where the first two equalities follow from Lemma \ref{basic}; the first equality also follows directly from the proof of \cite[Lemma 2.3]{IT2019}.

Next, we prove the second equality.  The inclusion $V(\ca_R(M))\subseteq V(\ca_R^{d+1}(M))$ is clear. Combining with the first equality, it remains to prove $V(\ca_R^{d+1}(M))\subseteq \IPD(M)$. If not, let $\p$ be a prime ideal containing $\ca_R^{d+1}(M)$ and $\p\notin \IPD(M)$. By assumption, $\pd_{R_\p}(M_\p)\leq d$. However, it follows from Lemma \ref{basic} that $\pd_{R_\p}(M_\p)\geq d+1$. This is a contradiction. 
\end{proof}

Recall that $\CM(R)$ is denoted to be the category of maximal Cohen--Macaulay $R$-modules (see \ref{def of MCM}). Set $$\OCM^\times(R)\colonequals \{M\in \Omega^1_R(\CM(R))\mid M \text{ has no nonzero projective summands}\};$$
see the definition of $\Omega^1_R(\CM(R))$ in \ref{def of syzygy}. Next, we calculate the uniform degree of the cohomological annihilators. It turns out that the uniform degree behaves well when the ring is strongly Gorenstein; see Proposition \ref{relation CA(R) and ca(R-mo)}.

\begin{lemma}\label{ca stablize}
    (1) Let $M$ be a finitely generated $R$-module such that $\Ext^{i}_R(M,R)=0$ for all $0\leq m\leq i\leq n$, then 
    $
    \ca^m_R(M)=\cdots=\ca^{n+1}_R(M)
    $
    
    (2) If $(R,\m)$ is a $d$-dimensional Cohen--Macaulay local ring with a canonical module $\omega$. If $\OCM^\times(R)$  is closed under $\Hom_R(-,\omega)$, then
    $
    \ca_R^{d+1}(M)=\ca_R(M)
    $
   for each $M\in \mo(R)$. In particular,
   $
\CA^{d+1}(R)=\CA(R)=\ca_R(\mo(R)).
$
\end{lemma}

\begin{proof}
    (1) For each $j\geq 0$, consider the short exact sequence
    $$
0\rightarrow \Omega_R^{j+1}(M)\rightarrow P\rightarrow \Omega_R^{j}(M)\rightarrow 0,
$$
where $P$ is a finitely generated projective $R$-module.  Applying $\Hom_R(M,-)$ to this short exact sequence, we get an exact sequence
$$\Ext^{j}_{R}(M, P) \to \Ext^{j}_{R}(M, \Omega^{j}_{R} (M)) \to \Ext^{j+1}_{R}(M, \Omega^{j+1}_{R} (M)).$$
Set $j=m$. The assumption yields that $\Ext^m_R(M,P)=0$. Thus, the above exact sequence yields that $\ca^{m+1}_R(M)\subseteq \ca^m_R(M)$ (see Lemma \ref{basic}), and hence $\ca^m_R(M)=\ca^{m+1}_R(M)$. The desired result follows similarly by choosing $j=m+1,\ldots,n$.

(2) For each $M\in \mo(R)$ and $i\geq d$, $\Omega^{i}_R(M)$ is in $\CM(R)$. Choose a minimal free resolution of $\Omega^i_R(M)$:
$
0\rightarrow \Omega^{i+1}_R(M)\rightarrow F\rightarrow \Omega^i_R(M)\rightarrow 0,
$
where $F$ is a finitely generated free $R$-module. In this case, $\Omega^{i+1}_R(M)$ does not have a nonzero projective summand (see \cite[Lemma 1.4]{Herzog1978}), and hence $\Omega^{i+1}_R(M)\in\OCM^\times(R)$.
By assumption, $\Hom_R(\Omega^{i+1}_R(M),\omega)\in \OCM(R)$. It follows from \cite[Proposition 5.2]{K-T} that there exists a short exact sequence
$$
0\rightarrow K\rightarrow \omega^{\oplus s}\rightarrow \Omega^{i+1}_R(M)\rightarrow 0
$$
in $\mo(R)$.
Note that $\Ext^{>0}_R(\Omega^i_R(M),\omega)=0$ as $\Omega^i_R(M)\in \CM(R)$. Applying $\Hom_R(\Omega^i_R(M),-)$ to the above short exact sequence, 
$
 \Ext^1_R(\Omega^i_R(M),\Omega^{i+1}_R(M))\cong \Ext^2_R(\Omega^i_R(M),K).
$
We conclude from this that $\Ext^{i+1}_R(M,\Omega^{i+1}_R(M))\cong \Ext^{i+2}_R(M,K)$. 
This isomorphism implies 
$\ca_R^{i+2}(M)\subseteq \ann_{R} \Ext^{i+1}_R(M,\Omega^{i+1}_R(M))=\ca_R^{i+1}(M),
$
where the second equality follows from Lemma \ref{basic}. Hence, $\ca_R^{i+1}(M)=\ca_R^{i+2}(M)$ for each $i\geq d$. This implies $\ca_R(M)=\ca_R^{d+1}(M)$.

For the second statement, it follows from Lemma \ref{test regular} that $\CA^{d+1}(R)\subseteq \CA(R)\subseteq \ca_R(\mo(R))$. The desired result now follows from the following:
$$
\CA^{d+1}(R)=\bigcap_{M\in \mo(R)}\ca_R^{d+1}(M)=\bigcap_{M\in \mo(R)}\ca_R(M)=\ca_R(\mo(R)),
$$
This finishes the proof.
\end{proof}
\begin{chunk}\label{Rfd}
    For a finitely generated $R$-module $M$, let $\Rfd_R (M)$ denote the \emph{(large) restricted flat dimension} of $M$. That is, 
$$
\Rfd_R (M) 
= \sup \{\depth (R_{\p})-\depth (M_{\p}) \mid
\p \in \Spec R \}.
$$
By \cite[Theorem 1.1]{AIL} with \cite[Proposition 2.2 and Theorem 2.4]{CFF}, 
$\Rfd_R (M) \in \mathbb{N} \cup \{-\infty \}$, and $\Rfd_R(M)=-\infty$ if and only if $M=0$.
\end{chunk}
 
\begin{proposition}\label{relation CA(R) and ca(R-mo)}
        Let $R$ be a Gorenstein ring. Then:
\begin{enumerate}
    \item For each $M\in \CM(R)$, $\ca_R^1(M)=\ca_R(M)$.

    \item For each nonzero finitely generated $R$-module $M$, $\ca_R^{l+1}(M)=\ca_R(M)$, where $l=\Rfd_R(M)$. 
    \item If, in addition, $R$ has finite Krull dimension $d$, then there are equalities 
$$\CA^{d+1}(R)=\CA(R)=\ca_R(\mo(R))=\bigcap_{M\in \CM(R)}\ann_R\underline{\End}_R(M).$$
\end{enumerate}
\end{proposition}

\begin{proof}
    (1) For each $M\in \CM(R)$, we have $\Ext^{>0}_R(M,R)=0$; see \ref{def of MCM}. The desired result now follows from Lemma \ref{ca stablize}.

    (2) Note that \ref{Rfd} yields that $l=\Rfd_R(M)$ is finite. By \cite[1.3.7]{BH}, $\Omega^l_R(M)$ is maximal Cohen--Macaulay over $R$. 
    Let $i\geq 1$ be an integer. By Lemma \ref{basic}, we have
    $$
\ca_R^{l+i}(M)=\ann_R\underline{\End}_R(\Omega^{l+i-1}_R(M)).
$$
Note that $\underline{\End}_R(\Omega^{l+i-1}_R(M))\cong \underline{\End}_R(\Omega^l_R(M))$ as $\Omega^{l}_R(M)\in \CM(R)$; see \cite[Theorem B.1.8]{Buchweitz}. Thus, $\ca_R^{l+i}(M)=\ca_R^{l+1}(M)$, and hence $
\ca_R(M)=\ca_R^{l+1}(M)$. 

(3) In \cite[Lemma 2.3]{Esentepe}, Esentepe observed that $\CA(R)=\displaystyle\bigcap\limits_{M\in \CM(R)}\ann_R\underline{\End}_R(M).$
By the same argument in the proof of Lemma \ref{ca stablize}, we have $\CA^{d+1}(R)=\CA(R)=\ca_R(\mo(R))$. This completes the proof.
\end{proof}

\begin{remark}\label{n(mo(R))}
    (1) Keep the same assumption as Proposition \ref{relation CA(R) and ca(R-mo)} (3), the first author and the second author proved $\CA(R)=\CA^{d+1}(R)$ in \cite[Proposition 3.4]{DeyLiu}. As mentioned in the proof, the equality $\CA(R)=\displaystyle\bigcap_{M\in \CM(R)}\ann_R\underline{\End}_R(M)$ was due to Esentepe \cite[Lemma 2.3]{Esentepe}. 

    (2)  For a finitely generated $R$-module $M$, we define $$\operatorname{n}_R(M) = \inf \{ m \geq 0 \mid \ca^{m+1}_{R}(M) = \ca_{R}(M) \}.$$ Since $R$ is Noetherian, $\operatorname{n}_R(M)$ is always finite. For instance, if $\pd_{R}(M)$ is finite, then $\operatorname{n}_R(M) = \pd_R(M)$. For a full subcategory $\mathcal C$ of $\mo(R)$, we set $\operatorname{n}_R(\mathcal C) = \sup \{ \operatorname{n}_R(M) \mid M\in \mathcal C \}$.  If $\operatorname{n}_R(\mathcal C)<\infty$, then $  \ca_{R}(\mathcal C) = \displaystyle\bigcap_{M\in\mathcal C}\ca_{R}(M) = \displaystyle\bigcap_{M\in\mathcal C}\ca^{k+1}_{R}(M)$ for each $k \geq \operatorname{n}_R(\mathcal C)$. Hence, if $\operatorname{n}_R(\mo(R))<\infty$ is finite, then 
    $$\ca_{R}(\mo(R)) = \bigcup_{k \geq \operatorname{n}_R(\mo(R))} (\bigcap_{M \in \mo(R)} \ca^{k+1}_{R}(M)) =\mathfrak{ca}(R).$$
 Note that the ring  $R$, under the assumption of Lemma \ref{ca stablize} (2) or Proposition \ref{relation CA(R) and ca(R-mo)} (2), satisfies \( \operatorname{n}_R(\mo(R)) \) is finite.

 (3) Let $(R,\m)$ be a commutative Noetherian local ring with $\m^2=0$. For each $M\in\mo(R)$, the assumption that $\m^2=0$ implies that $\Omega_R^1(M)$ is a direct sum of some finite copies of $R$ and $k$. By Lemma \ref{basic} (1), we conclude that $\ca^2_R(M)=\ca_R(M)$ for each $M\in\mo(R)$, and hence $\operatorname{n}_R(\mo(R))\leq 1$. 
\end{remark}
Let $\Max(R)$ denote the subset of $\Spec(R)$ consisting of all maximal ideals of $R$. In view of Remark \ref{n(mo(R))} (2), the following result may be regarded as a version of the local-global principle for projective dimensions.
\begin{proposition}\label{local-global-n(mo(R))}
Let $R$ be a commutative Noetherian ring. Then: 

\begin{enumerate}
\item For each  $M \in\mo(R)$, $\operatorname{n}_R(M)=\sup\{\operatorname{n}_{R_\p}(M_\p)\mid \p \in \Spec(R)\}
=\sup\{\operatorname{n}_{R_\m}(M_\m)\mid \m \in \Max(R)\}$. In particular, $\operatorname{n}_R(\mo(R))  =\sup\{\operatorname{n}_{R_\p}(\mo(R_\p)) \mid \p \in \Spec(R)\}=\sup\{\operatorname{n}_{R_\m}(\mo(R_\m)) \mid \m \in \Max(R)\}$. 

    \item $\operatorname{n}_R(\mo(R))\geq \dim(R) - 1$. In particular, if $\operatorname{n}_R(\mo(R))$ is finite, then $R$ has finite Krull dimension. 
\end{enumerate}
\end{proposition}
\begin{proof}
(1) The second statement follows from the first one and the fact that the localization functor $\mo(R)\rightarrow\mo(R_\p)$ is dense for each $\p\in\Spec(R)$. For the first statement, let $n=\operatorname{n}_R(M)+1$. It follows that $\ca^n_R(M)=\ca_R(M)$. By Lemma \ref{basic} (5), for each $\p\in\Spec(R)$, we have
$$\ca^n_{R_\p}(M_\p)=\ca_R^n(M)_\p=\ca_R(M)_\p=\ca_{R_\p}(M_\p).$$
It follows form this that $\operatorname{n}_R(M)\geq \sup\{\operatorname{n}_{R_\p}(M_\p)\mid \p \in \Spec(R)\}
\geq\sup\{\operatorname{n}_{R_\m}(M_\m)\mid \m \in \Max(R)\}$. Assume $m=\sup\{\operatorname{n}_{R_\m}(M_\m)\mid \m \in \Max(R)\}+1$.  Again by Lemma \ref{basic} (5), we conclude that the localization of the inclusion $\ca^m_R(M)\subseteq \ca_R(M)$ at each maximal ideal of $R$ is equal, and hence $\ca^m_R(M)=\ca_R(M)$. It follows from this that $\operatorname{n}_R(M)\leq m=\sup\{\operatorname{n}_{R_\m}(M_\m)\mid \m \in \Max(R)\}$.

 (2)  Let $\p \in \Spec(R)$, and let $J$ be an ideal of $R_\p$ generated by a maximal regular sequence of $R_\p$. 
Since $J$ has finite projective dimension over $R_\p$, we obtain the first equality below
\[
\operatorname{n}_{R_\p}(R_\p/J)=\pd_{R_\p}(R_\p/J)=\depth(R_\p)
\]
where the second equality is by our choice of $J$. 
Combining with (1), we conclude that
$
\operatorname n_R(\mo(R)) \geq \sup \{ \depth (R_\p) \mid \p \in \Spec(R) \} \geq \dim(R) - 1,
$
where the second inequality is by \cite[Lemma 1.4]{CFF}.
\end{proof}

\section{(Strong) generation of singularity categories}\label{Section singularity cat}
In this section, we investigate the (strong) generation of the singularity category. The main result of this section is Theorem \ref{T1} from the introduction; see Theorems \ref{main} and \ref{main-second}.
\begin{chunk}\label{def of ann of sin cat}
    Let $X$ be a complex in the singularity category $\D_{\sg}(R)$. The \emph{annihilator} of $X$ over $\D_{\sg}(R)$, denoted $\ann_{\D_{\sg}(R)}(X)$, is defined to be the annihilator of $\Hom_{\D_{\sg}(R)}(X,X)$ over $R$. That is, 
    $$
\ann_{\D_{\sg}(R)}(X)\colonequals \{r\in R\mid r\cdot \Hom_{\D_{\sg}(R)}(X,X)=0\}.
$$
The annihilator of $\D_{\sg}(R)$ is defined to be 
$$
\ann_R\D_{\sg}(R)\colonequals \bigcap_{X\in \D_{\sg}(R)}\ann_{\D_{\sg}(R)}(X);
$$
see \cite{Liu} for more details about the annihilator of the singularity category.
\end{chunk}
\begin{proposition}\label{compare}
     Let $M$ be a finitely generated $R$-module. Then
    $$
    \ca_R(M)=\ann_{\D_{\sg}(R)}(M).
    $$
\end{proposition}
\begin{proof}
First, we prove the $\ca_R(M)\subseteq \ann_{\D_{\sg}(R)}(M)$. Let $r\in \ca_R^1(M)$. Note that $\ca_R^1(M)=\ann_R\underline{\End}_R(M)$; see Lemma \ref{basic}. Thus, the multiplication map $r\colon M\rightarrow M$ factors through a projective $R$-module. In particular, the multiplication map $r\colon M\rightarrow M$ in the singularity category $\D_{\sg}(R)$ is zero. This yields that $r\in \ann_{\D_{\sg}(R)}(M)$, and hence $\ca_R^1(M)\subseteq \ann_{\D_{\sg}(R)}(M)$.
For each $n>1$,
\begin{align*}
    \ca_R^n(M) =\ca_R^1(\Omega^{n-1}_R(M))\subseteq \ann_{\D_{\sg}(R)}(\Omega^{n-1}_R(M)),
\end{align*}
where the equality is by Lemma \ref{basic}, and the inclusion follows from the argument above. By the isomorphism $\Omega^{n-1}_R(M)[n-1]\cong M$, we have $$\ann_{\D_{\sg}(R)}(\Omega^{n-1}_R(M))=\ann_{\D_{\sg}(R)}(M).$$ Thus, $\ca_R^n(M)\subseteq \ann_{\D_{\sg}(R)}(M)$ for all $n>1$, and hence $\ca_R(M)\subseteq \ann_{\D_{\sg}(R)}(M).$

Now, we prove the converse direction. Assume $a\in \ann_{\D_{\sg}(R)}(M)$. That is, the multiplication map $a\colon M\rightarrow M$ in $\D_{\sg}(R)$ is zero. This implies the multiplication map $a\colon M\rightarrow M$ in $\D^f(R)$ factors through a perfect complex $P$. That is, in $\D^f(R)$, there is a commutative diagram
$$
\xymatrix{
M\ar[rr]^-a\ar[rd]&& M\\
& P\ar[ru]& 
}.
$$
Since $P$ is perfect, there exists $n>0$ such that 
$$
\Hom_{\D^f(R)}(P,N[i])=0
$$
for all $i\geq n$ and $N\in \mo(R)$. Applying $\Hom_{\D^f(R)}(-,N[i])$ on the above commutative diagram, we get that the multiplication map
$$
a\colon \Hom_{\D^f(R)}(M,N[i])\rightarrow \Hom_{\D^f(R)}(M,N[i])
$$
is zero for all $i\geq n$ and $N\in \mo(R)$. Combining with the isomorphism $$\Ext^i_R(M,N)\cong \Hom_{\D^f(R)}(M,N[i]),$$ we conclude that $a\cdot \Ext^i_R(M,N)=0$ for all $i\geq n$ and $N\in \mo(R)$. That is, $a\in \ca_R^n(M)$, and hence $a\in \ca_R(M)$. It follows that $\ann_{\D_{\sg}(R)}(M)\subseteq \ca_R(M)$. 
\end{proof}
\begin{corollary}\label{annihilator of sin}
     Let $R$ be a commutative Noetherian ring. Then
     $$
     \ann_R\D_{\sg}(R)=\ca_R(\mo(R)).
     $$
\end{corollary}
\begin{proof}
    Let $X$ be a complex in $\D_{\sg}(R)$. By choosing a projective resolution of $X$, we may assume $X$ is a bounded
above complex of finitely generated projective $R$-modules with finitely many nonzero
cohomologies. Then by taking brutal truncation, we conclude that $X[n]$ is isomorphic
to a finitely generated $R$-module in $\D_{\sg}(R)$ for $n\ll 0$. Assume $X[n]\cong M$ for some $n\ll 0$ and $M\in \mo(R)$. Note that $\ann_{\D_{\sg}(R)}(X)=\ann_{\D_{\sg}(R)}(M)$. This yields the second equality below:
\begin{align*}
    \ann_R\D_{\sg}(R)&=\bigcap_{X\in \D_{\sg}(R)}\ann_{\D_{\sg}(R)}(X)\\
    &=\bigcap_{M\in \mo(R)}\ann_{\D_{\sg}(R)}(M)\\
    &=\bigcap_{M\in \mo(R)}\ca_R(M),
\end{align*}
where the third equality follows from Proposition \ref{compare}.
\end{proof}
Let $\Sing(R)$ denote the \emph{singular locus} of $R$. That is, $\Sing(R)\colonequals \{\p\in \Spec(R)\mid R_\p \text{ is not regular}\}$.
\begin{lemma}\label{V(ca)=Sing(R)}
   For each finitely generated $R$-module $M$, $V(\ca_R(M))\subseteq \Sing(R).
$   The equality holds if, in addition, $M$ is a generator of $\D_{\sg}(R)$.
\end{lemma}
\begin{proof}
  It follows from Proposition \ref{relation} that $V(\ca_R(M))=\IPD(M)$. It is clear that $$\IPD(M)\subseteq \Sing(R).$$   By \cite[Lemma 2.6]{IT2016}, the equality holds if $M$ is a generator of $\D_{\sg}(R)$.
\end{proof}

\begin{chunk}\label{equi def of strong gen-abel}
    For each $X\in \mo(R)$, Dao and Takahashi \cite[Defintion 5.1]{DT2014} introduced a construction of an ascending chain of full subcategories built out of $X$ as follows: 

Set $|X|_0^R=\{0\}$, and denote $|X|_1^R$ to be the full subcategory of $\mo(R)$ consisting of modules which are direct summands of any finite copies of $X$. Inductively, for $n\geq 2$, denote $|X|_n^R$ to be the full subcategory of $\mo(R)$ consisting of modules $Y$ that appear in a short exact sequence
$$
0\rightarrow Y_1\rightarrow Y\oplus Y^\prime \rightarrow Y_2\rightarrow 0,
$$
where $Y_1\in |X|_{n-1}^R$ and $Y_2\in |X|_1^R$. If there is no confusion, we will use $|X|_n$ to denote $|X|_n^R$.

In \cite[Corollary 4.6]{IT2016}, Iyengar and Takahashi observed that $\mo(R)$ has a strong generator if and only if there exist $s,t\geq 0$ and $G\in \mo(R)$ such that 
$$
\Omega^s_R(\mo(R))\subseteq |G|_t.
$$
\end{chunk}

\begin{chunk}\label{R/m^i}
    Let $(R,\m)$ be a commutative Noetherian local ring. Assume there exist $s,t\geq 0$ and $G\in\mo(R)$ such that $\Omega^s_R(R/\m^i)\subseteq [G]_t^R$ for all $i\geq 0$; see the definition of $[G]_t^R$ (i.e., the \emph{ball of radius} $t$ \emph{centered at} $G$) in \cite[Definition 2.1]{DT2014}. With the same argument in the proof of \cite[Proposition 3.2]{ST2021}, one has $\dim(R)\leq s$.
\end{chunk}
The following proposition characterizes the Krull dimension of $R$ in terms of strong generation of $\mo(R)$.
\begin{proposition}\label{Krull dim finite}
Let $R$ be a commutative Noetherian ring. 
Then:
\[
\dim(R)\leq\inf\{s\geq 0\mid \text{there exist } t\geq 0,~G\in\mo(R) \text{ such that }\Omega^s_R(\mo(R))\subseteq |G|_t^R\}.
\]
In particular, if $\mo(R)$ has a strong generator, then $\dim(R)$ is finite.
In addition, if $R$ is quasi-excellent and either admits a dualizing complex or is a local ring, then the equality holds.
\end{proposition}

\begin{proof}
Let $s,t$ be nonnegative integers and $G$ be a finitely generated $R$-module. Assume $\Omega^s_R(\mo(R))\subseteq |G|_t^R$. 
Since $|G|_t^R\subseteq [G]_t^R$, $\Omega_{R}^{s}(\mo(R))\subseteq {\lbrack G\rbrack}_{t}^{R}$; c.f. \ref{R/m^i}.
Let $\p$ be a prime ideal of $R$. We claim that 
 $\Omega_{R_{\p}}^{s}(\mo (R_{\p}))\subseteq {\lbrack G_{\p}\rbrack}_{t}^{R_{\p}}.$
Combining with this, \ref{R/m^i} yields that $\dim(R_\p)\leq s$, and hence $\dim(R) \leq s$ as $\p$ is arbitrary.

Now, we proceed to prove the claim. 
Let $X$ be an object in $\Omega_{R_{\p}}^{s}(\mo(R_{\p}))$. Since the functor $(-)_{\p}\colon \mo(R)\to\mo(R_{\p})$ is dense, there exists $M\in\mo(R)$ such that $X\cong\Omega_{R_{\p}}^{s}(M_{\p})$. On the other hand, there exist $n,m\geq0$ such that $\Omega_{R_{\p}}^{s}(M_{\p})\oplus R_{\p}^{\oplus n}\cong \left(\Omega_{R}^{s}(M)\right)_{\p}\oplus R_{\p}^{\oplus m}$. Hence,  $$X\oplus R_{\p}^{\oplus n}\cong \left(\Omega_{R}^{s}(M)\right)_{\p} \oplus R_{\p}^{\oplus m}\cong \left(\Omega_{R}^{s}(M) \oplus R^{\oplus m}\right)_\p\in \left(\Omega_{R}^{s}(\mo(R))\right)_{\p}\subseteq {\lbrack G_{\p}\rbrack}_{t}^{R_{\p}};$$
the inclusion here uses
 $
 \left(\Omega_{R}^{s}(\mo(R))\right)_{\p}\subseteq ( [G]_{t}^{R})_{\p}\subseteq {\lbrack G_{\p}\rbrack}_{t}^{R_{\p}}.$ As ${\lbrack G_{\p}\rbrack}_{t}^{R_{\p}}$ is closed under direct summands, one has $X\in{\lbrack G_{\p}\rbrack}_{t}^{R_{\p}}$.


Suppose that $R$ is a quasi-excellent ring of finite Krull dimension $d$ and either admits a dualizing complex or is a local ring. 
Combining \cite[Theorem 5.2]{IT2016} and \cite[Corollary 2.6]{Kimura2024}, there exist $t\geq0$ and $G\in\mo(R)$ such that  $\Omega^d_R(\mo(R))\subseteq |G|_t^R$. This implies that the converse of the inequality stated in the proposition holds.
\end{proof}
\begin{remark}
Let $R$ be a commutative Noetherian ring, and let $\Max(R)$ denote the set of all maximal ideals of $R$. 
By a similar argument as in the previous proposition, the following three values coincide.
\begin{enumerate}
\item
The Krull dimension of $R$.
\item
\begin{equation*}
\inf \left\{ s \geq 0 \, \middle| 
\begin{array}{l}
\text{for every } \p \in \Spec(R), \text{ there exist } t \geq 0, ~G \in \mo(\widehat{R_{\p}}) \\
\text{such that } 
\Omega^s_{\widehat{R_{\p}}}(\mo(\widehat{R_{\p}})) \subseteq |G|_t^{\widehat{R_{\p}}}
\end{array}
\right\}.
\end{equation*}
\item
\begin{equation*}
\inf \left\{ s \geq 0 \, \middle| 
\begin{array}{l}
\text{for every } \m \in \Max(R), \text{ there exist } t \geq 0, ~G \in \mo(\widehat{R_{\m}}) \\
\text{such that } \Omega^s_{\widehat{R_{\m}}}(\mo(\widehat{R_{\m}})) \subseteq |G|_t^{\widehat{R_{\m}}}
\end{array}
\right\}.
\end{equation*}
\end{enumerate}
Additionally, in (2), the values obtained by replacing $\Omega^s_{\widehat{R_{\p}}}(\mo(\widehat{R_{\p}})) \subseteq |G|_t^{\widehat{R_{\p}}}$ with 
\begin{itemize}
\item
$\Omega^s_{\widehat{R_{\p}}}(\mo(\widehat{R_{\p}})) \subseteq \lbrack G\rbrack_t^{\widehat{R_{\p}}}$,
\item
$\Omega^s_{\widehat{R_{\p}}}(\widehat{R_{\p}}/\p^{i}\widehat{R_{\p}}) \in |G|_t^{\widehat{R_{\p}}} \text{ for all } i\gg0$, or 
\item
$\Omega^s_{\widehat{R_{\p}}}(\widehat{R_{\p}}/\p^{i}\widehat{R_{\p}}) \in \lbrack G\rbrack_t^{\widehat{R_{\p}}} \text{ for all } i\gg0$
\end{itemize}  
are all equal. The same holds for (3).
\end{remark}

The equality \(\Sing(R) = V(\CA(R))\) in (3) below was established by Iyengar and Takahashi \cite[Theorem 1.1]{IT2016} under the additional assumption that the finitistic global dimension is finite. Proposition \ref{Krull dim finite} shows this assumption is unnecessary.
\begin{proposition}\label{recover}
   Let $R$ be a commutative Noetherian ring. Then:
\begin{enumerate}
    \item If $\D_{\sg}(R)$ has a generator, then there exists $M\in \mo(R)$ such that $M$ is generator in $\D_{\sg}(R)$. Moreover, for any $M\in \mo(R)$ which is a generator in $\D_{\sg}(R)$,
$$
\Sing(R)=V(\ann_{\D_{\sg}(R)}(M))=V(\ca_R(M))=\IPD(M).
$$

\item If $\D_{\sg}(R)$ has a strong generator, then
$$
\Sing(R)=V(\ann_R{\D_{\sg}(R)})=V(\ca_R(\mo(R))).
$$

\item  If $\mo(R)$ has a strong generator, then
$$
\Sing(R)=V(\CA(R))=V(\ann_R{\D_{\sg}(R)})=V(\ca_R(\mo(R))).
$$

\item $\CA(R)\subseteq \ann_R\D_{\sg}(R)=\ca_R(\mo(R))$. If, in addition, $\mo(R)$ has a strong generator, then they are equal up to radical. 
\end{enumerate}
\end{proposition}
\begin{proof}

(1) Assume $G$ is a generator of $\D_{\sg}(R)$. With the same argument in the proof of Corollary \ref{annihilator of sin}, there exists $n\in \Z$ such that $G[n]\cong M$ in $\D_{\sg}(R)$ for some $M\in \mo(R)$. Then $M$ is also a generator of $\D_{\sg}(R)$. Combining with Lemma \ref{V(ca)=Sing(R)} and Proposition \ref{compare},
$$
\Sing(R)=V(\ca_R(M))=V(\ann_{\D_{\sg}(R)}(M)).
$$
The desired result follows as $V(\ca_R(M))=\IPD(M)$; see Proposition \ref{relation}.

(2) By Corollary \ref{annihilator of sin}, $\ann_R\D_{\sg}(R)=\ca_R(\mo(R))$. It remains to prove the first equality. Assume $\D_{\sg}(R)=\thick^s_{\D_{\sg}(R)}(G)$ for some $G\in \D_{\sg}(R)$ and $s\geq 0$. By definition, $\ann_R\D_{\sg}(R)\subseteq \ann_{\D_{\sg}(R)}(G)$. Combining with the assumption, it follows from \cite[Lemma 2.1]{Esentepe} that $$(\ann_{\D_{\sg}(R)}(G))^s\subseteq \ann_R\D_{\sg}(R).$$ 
Thus, $\ann_R\D_{\sg}(R)$ is equal to $\ann_{\D_{\sg}(R)}(G)$ up to radical. This yields that 
$$
V(\ann_R\D_{\sg}(R))=V(\ann_{\D_{\sg}(R)}(G)).
$$
The desired result now follows immediately from (1); note that $G[n]\cong M$ in $\D_{\sg}(R)$ for some $M\in \mo(R)$ and $n\in \Z$.

(3) The strong generation of $\mo(R)$ implies the strong generation of $\D_{\sg}(R)$; see \ref{strong gen regular finite dim}. Combining with (2), it remains to show the first equality, and this follows immediately by combining Proposition \ref{Krull dim finite} and \cite[Theorem 1.1]{IT2016}.

(4) The first statement follows from Corollaries \ref{test regular} and \ref{annihilator of sin}. The second statement follows from (3).
\end{proof}
\begin{remark}\label{difference}
(1) The equality $\Sing(R)=\IPD(M)$ in Proposition \ref{recover} (1) was due to Iyengar and Takahashi \cite[Lemma 2.9]{IT2019}. The new input in (1) is the defining ideal of $\Sing(R)$ via the annihilator over the singularity category. 

(2)  The first equality of Proposition \ref{recover} (2) was established by the second author in \cite[Theorem 4.6]{Liu} through the localization of the singularity category. In contrast, our proof takes a different approach by employing the cohomological annihilators of modules and their connection to the annihilators of the singularity category. The inclusion $\CA(R)\subseteq \ann_R\D_{\sg}(R)$ was observed in \cite[Proposition 1.2]{Liu}.

(3) Assume $\D_{\sg}(R)$ has a strong generator. In \cite[Corollary 5.8]{BSTT2022},  Bahlekeh, Salarian, Takahashi, and Toosi observed that
    $$
    \dim (R/\ca_R(\mo(R)))\leq \dim \Sing(R).
    $$
    Indeed, this inequality is an equality by Proposition \ref{recover} (2).

\end{remark}

\begin{corollary}\label{nonzero}
    Let $R$ be a commutative Noetherian domain. If $\D_{\sg}(R)$ has a strong generator, then $\ann_R\D_{\sg}(R)\neq 0$. 
\end{corollary}
\begin{proof}
    By Proposition \ref{recover}, $\ann_{R}\D_{\sg}(R)$ defines the singular locus of $R$. It follows from this that $\ann_{R}\D_{\sg}(R)\neq 0$. If not, the zero ideal of $R$, denoted $(0)$, is in $\Sing(R)$. This contradicts that $R_{(0)}$ is a field. Hence, $\ann_{R}\D_{\sg}(R)\neq 0$.
\end{proof}

\begin{remark}
(1) Corollary \ref{nonzero} is an analog of a result by Elagin and Lunts \cite[Theorem 5]{Elagin-Lunts:2018} concerning the derived category. Specifically, for a commutative Noetherian domain $R$,  they observed that $\CA(R)\neq 0$ if $\D^f(R)$ has a strong generator. 

 In general, if $\D^f(R)$ has a strong generator, then $\D_{\sg}(R)$ also possesses one. However, the converse is not true, even when $R$ is a domain; for instance, consider a regular domain with infinite Krull dimension constructed by Nagata \cite[Appendix, Example 1]{Nagata} (see \ref{strong gen regular finite dim}). Additionally, $\CA(R)\subseteq \ann_R\D_{\sg}(R)$ (see Proposition \ref{recover}), and the inclusion can be proper, as illustrated by Nagata's example (see Lemma \ref{test regular} and Corollary \ref{annihilator of sin}). Therefore, there seems to be no implication between Corollary \ref{nonzero} and the aforementioned result of Elagin and Lunts.

 (2) There exists a one-dimensional Noetherian domain with $\ann_R\D_{\sg}(R)=0$; see Example \ref{ann sin-cat zero}.
\end{remark}

\begin{example}\label{quasi-excellent}
    Let $R$ be a finitely generated algebra over a field or an equicharacteristic excellent local ring. Iyengar and Takahashi \cite[Theorem 1.3]{IT2016} established that $\mo(R)$ has a strong generator. 
Recently, this result was extended by the first author, Lank, and Takahashi \cite[Corollary 3.12]{DLT}, who demonstrated that $\mo(R)$ has a strong generator for any quasi-excellent ring with finite Krull dimension.
\end{example}

\begin{corollary}\label{equal up to radical}
  Let $R$ be a quasi-excellent ring with finite Krull dimension. Then
  $$
  \sqrt{\CA(R)}=\sqrt{\ann_R\D_{\sg}(R)}=\sqrt{\ca_R(\mo(R))}.
  $$
\end{corollary}
\begin{proof}
    This follows immediately from Proposition \ref{recover} and Example \ref{quasi-excellent}.
\end{proof}

\begin{remark}
     Assume $R$ is a localization of a finitely generated algebra over a field or an equicharacteristic excellent local ring. Then $R$ is excellent with finite Krull dimension, and hence Corollary \ref{equal up to radical} implies that $\CA(R)$ is equal to $\ann_R\D_{\sg}(R)$ up to radical; in this case, this was proved in \cite[Proposition 1.2]{Liu}.

\end{remark}

\begin{chunk}\label{up to syzygy}
Inspired by the observation of Iyengar and Takahashi in \ref{equi def of strong gen-abel}, we say that $\mo(R)$ has a \emph{point-wise strong generator} provided that there exist $G\in \mo(R)$ and $t\geq 0$ such that, for each $M\in \mo(R)$, $\Omega_{R}^{s}M\in \left|G\right|_{t}$ for some $s\geq 0$.
\end{chunk}

We say that $R$ has \emph{isolated singularities} if $\Sing (R)\subseteq \Max (R)$, where $\Max(R)$ consists of all maximal ideals of $R$.

\begin{theorem}\label{main}
Let $R$ be a commutative Noetherian ring with isolated singularities. Then the following are equivalent.

\begin{enumerate}
\item
 $\D_{\sg}(R)$ has a strong generator.
\item
 $R/\ann_{R}\D_{\sg}(R)$ is either $0$ or Artinian. 
\item
$R/\ca_{R}(\mo(R))$ is either $0$ or Artinian. 

\item
$\mo(R)$ has a point-wise strong generator.

\end{enumerate}
\end{theorem}
\begin{proof}
$(1)\Rightarrow(2)$: By Proposition \ref{recover}, one has $\Sing (R)=V(\ann_{R}\D_{\sg}(R))$. Since $R$ has isolated singularities, $\Sing(R)$ consists of maximal ideals.
Hence, $V(\ann_{R}\D_{\sg}(R))$ consists of maximal ideals. This implies $(2)$.

$(2)\Rightarrow (3)$: This follows from Corollary \ref{annihilator of sin}.

$(3)\Rightarrow(4)$:
If $R/\ca_R(\mo(R))=0$, then $R$ is regular; see Lemma \ref{test regular}. Thus, for each $M\in\mo(R)$, $\Omega^i_R(M)$ is finitely generated projective for some $i\geq 0$, and hence $\Omega^i_R(M)\in|R|_1$. This implies that $\mo(R)$ has a point-wise strong generator.  Next, we assume that $R/\ca(\mo(R))$ is Artinian. 

Set $I\colonequals\ca_R(\mo(R))$ and assume $\x=x_1,\ldots,x_n$ is a generating set of $I$. Let $M$ be a finitely generated $R$-module. 
Then
$$I=\ca_{R}(\mo(R))\subseteq\ca_{R}(M)=\ca_{R}^{i+1}(M)=\ca_{R}^{1}(\Omega_{R}^{i}(M))$$
for some $i\geq  0$, where the second equality is by Lemma \ref{basic}. This implies that
$$
I\subseteq \ann_R\Ext^1_R(\Omega^i_R(M),\Omega^{i+1}_R(M)).
$$
Combining with this, \cite[Theorem 3.1 (1)]{BHST} shows that there exists $L\in\mo(R/I)$ such that
$
\Omega^i_R(M)\in |\displaystyle\bigoplus_{j=1}^n\Omega_R^j(L)|^R_{n+1}.
$
Since $R/I$ is Artinian, $L\in |N|^{R/I}_{\ell\ell(R/I)}$, where $\ell\ell(R/I)$ is the Loewy length of $R/I$ and $N$ is the quotient of $R/I$ by its Jacobson radical.
Restriction scalars along the map $R\rightarrow R/I$, one has $L\in |N|^R_{\ell\ell(R/I)}$, and hence $\Omega^j_R(L)\subseteq |\Omega^j_R(N)|^R_{\ell\ell(R/I)}$ for each $j\geq 0$. It follows from this that $\displaystyle\bigoplus_{j=1}^n\Omega^j_R(L)\in |\displaystyle\bigoplus_{j=1}^n\Omega^j_R(N)|^R_{\ell\ell(R/I)}$. By \cite[Lemma 3.8]{DLT}, $\Omega^i_R(M)\in |\displaystyle\bigoplus_{j=1}^n\Omega^j_R(N)|^R_{(n+1)\ell\ell(R/I)}$. Hence, $\displaystyle\bigoplus_{j=1}^n\Omega^j_R(N)$ is a point-wise strong generator of $\mo(R)$. 


$(4)\Rightarrow (1)$: This can be deduced by \cite[Proposition 5.3]{BSTT2022}.
\end{proof}

\begin{remark}\label{compare with BHST}
(1) The implication $(2)\Rightarrow (1)$ of the above theorem follows from the third author’s work \cite[Corollary 4.3]{Mifune}.

(2) Let $R$ be a commutative Noetherian ring with isolated singularities, if $\D_{\sg}(R)$ has a strong generator and $\operatorname{n}(\mo(R))<\infty$, then $\mo(R)$ has a strong generator. Indeed, this follows from the same argument as in the proof of Theorem \ref{main} $(3)\Rightarrow (4)$.  

Combining the above with Proposition \ref{relation CA(R) and ca(R-mo)}, if $R$ is a Gorenstein ring with isolated singularities and $R$ is of finite Krull dimension, then the following are equivalent.
\begin{itemize}
   \item $\D_{\sg}(R)$ has a strong generator. 
    \item $R/\CA(R)$ is either $0$ or Artinian.
    \item $\mo(R)$ has a strong generator.
\end{itemize}
The local case of this was established by Bahlekeh, Hakimian, Salarian, and Takahashi \cite[Theorem 3.2]{BHST}.


\end{remark}


\begin{corollary}\label{new class}
    Let $R$ be a one-dimensional Cohen--Macaulay ring. Assume that, for each maximal ideal $\m$ of $R$ of height $0$, the local ring $R_\m$ has minimal multiplicity. Additionally, assume that, for each maximal ideal $\m$ of height $1$, $R_\m$ is almost Gorenstein.
     If $\D_\sg(R)$ has a strong generator, then we have:

\begin{enumerate}
    \item $\Sing(R)=V(\CA(R))$.

    \item If, in addition, $R$ has isolated singularities, then $\mo(R)$ has a strong generator. 
\end{enumerate}
\end{corollary}
\begin{proof}
   (1) We claim that $\operatorname{n}_R(\mo(R))<\infty$. By Proposition \ref{local-global-n(mo(R))}, this is equivalent to showing that $\sup\{\operatorname{n}_{R_\m}(\mo(R_\m)\mid \m\in\Max(R)\}<\infty$. If $\m$ is a maximal ideal of height $0$, the by hypothesis $\m^2R_\m=0$, and hence $\operatorname{n}_{R_\m}(\mo(R_\m)) \leq 1$ by Remark \ref{n(mo(R))} (3). If $\m$ is a maximal ideal height $1$, then, if we can show $\syz \CM^{\times}(R_\m)$ is closed under canonical dual, we can conclude by Lemma \ref{ca stablize}  that $\operatorname{n}_{R_\m}(\mo (R_\m)) \leq 1$. So we now verify that for any one-dimensional almost Gorenstein local ring $R$, $\syz \CM^{\times}(R)$ is closed under canonical dual. Indeed, if $R$ is almost Gorenstein, then $\m^{\dagger}\in \ul_{\omega}(R)$ by \cite[Proposition 7.2, Lemma 4.15]{dms}. Notice that for any $M\in \syz \CM^{\times}(R)$, $M\cong \Hom_R(M^*, R)=\Hom_R(M^*,\m)\cong \Hom_R(\m^{\dagger}, (M^*)^{\dagger})$, where the last module belongs to $\ul_{\omega}(R)$ by \cite[Lemma 4.15]{dms}. Thus, by \cite[Corollary 4.27]{dms} we get $\Hom_R(M,\omega)\cong M^*\in \syz \CM^{\times}(R)$ as we wanted.    Therefore, $\operatorname{n}_R(\mo(R))\leq 1$. Combining this with Remark \ref{n(mo(R))} (2), we have $\CA(R)=\ca_R(\mo(R))$. Since $\D_{\sg}(R)$ has a strong generator, it follows from Proposition \ref{recover} that $\Sing(R)=V(\CA(R))$.

   (2) By the proof of (1), $\operatorname{n}_R(\mo(R))<\infty$. Remark \ref{compare with BHST} now shows that $\mo(R)$ has a strong generator. 
\end{proof}
\begin{remark}\label{relation-3.2-3.3}
   For a Gorenstein local ring $R$, Bahlekeh, Hakimian, Salarian, and Takahashi \cite[Theroem 3.3]{BHST} observed that if $\D_{\sg}(R)$ has a strong generator, then $\Sing(R)=V(\CA(R))$. Thus, Corollary \ref{new class} provides a new class of rings satisfying the equality $\Sing(R)=V(\CA(R))$. Moreover, the rings satisfying Corollary \ref{new class} (2) form a new class of rings for which all the conditions of \cite[Theorem 3.2]{BHST} are equivalent.
\end{remark}
The following result is a direct consequence of Theorem \ref{main}.
\begin{corollary}\label{main application}
 Let $R$ be a commutative Noetherian domain with $\dim R\leq 1$. Then the following are equivalent.
 \begin{enumerate}
     \item $\D_{\sg}(R)$ has a strong generator. 

     \item $\ann_R\D_{\sg}(R)\neq 0$.

\item $\ca_R(\mo(R)) \neq 0$.

     \item $\mo(R)$ has a point-wise strong generator.
 \end{enumerate}    
\end{corollary}

\begin{corollary}\label{sin version}
     Let $R$ be a commutative Noetherian ring with $\dim R\leq 1$. Then the following are equivalent.
\begin{enumerate}
 \item $\D_{\sg}(R/\p)$ has a strong generator for each prime ideal $\p$ of $R$.
 
    \item $\ann_{R/\p}\D_{\sg}(R/\p)\neq 0$ for each prime ideal $\p$ of $R$.

\item $\ca_{R/\p}(\mo(R/\p))\neq 0$ for each prime ideal $\p$ of $R$.   
\end{enumerate}
\end{corollary}

\begin{proof}
For each prime ideal $\p$ of $R$, $R/\p$ is a domain and $\dim R/\p\leq 1$. Thus, the equivalence $(1)\iff  (2)$  follows from Corollary  \ref{main application}. The equivalence $(2)\iff (3)$ is a direct consequence of Corollary \ref{annihilator of sin}.
\end{proof}
\begin{remark}\label{Krull dimension assumption}
Due to Iyengar and Takahashi \cite[Theorem 1.1]{IT2019}, the following three conditions are equivalent for a commutative Noetherian ring.
\begin{enumerate}
    \item $\mo(R/\p)$ has a generator for each prime ideal $\p$ of $R$.   

    \item $\D^f(R/\p)$ has a generator for each prime ideal $\p$ of $R$. 

    \item $\D_{\sg}(R/\p)$ has a generator for each prime ideal $\p$ of $R$. 
\end{enumerate}
Moreover, if these conditions hold, then $\mo(R)$, $\D^f(R)$, and $\D_{\sg}(R)$ all have generators. 

Recently, the first author, Lank, and Takahashi \cite[Theorem 1.1]{DLT} showed that the following three conditions are equivalent for a commutative Noetherian ring.
\begin{enumerate}
    \item $\mo(R/\p)$ has a strong generator for each prime ideal $\p$ of $R$.   

    \item $\D^f(R/\p)$ has a strong generator for each prime ideal $\p$ of $R$. 

    \item $\CA(R/\p)\neq 0$ for each prime ideal $\p$ of $R$.
\end{enumerate}
Furthermore, if these conditions hold, then $\dim(R)<\infty$, and both $\mo(R)$ and $\D^f(R)$ have strong generators. 

A natural question arises: do Corollaries \ref{main application} and \ref{sin version} still hold if we remove the assumption about the Krull dimension? Theorem \ref{main-second} is a result related to this question. Keep the same assumption as Corollary \ref{sin version}, we don't know whether $\D_{\sg}(R)$ has a strong generator when the conditions of Corollary \ref{sin version} hold. 
\end{remark}

For a full subcategory $\mathcal C$ of $\mo(R)$, we define $\sqrt{\ca}_{R}(\mathcal C)\colonequals\displaystyle\bigcap_{X\in\mathcal C}\sqrt{\ca_{R}(X)}$.

\begin{lemma}\label{syzygy}
Let $R$ be a commutative Noetherian ring. Then:

(1) $\mathfrak{ca}(R)\subseteq\ca(\mo(R))\subseteq\sqrt{\ca}_{R}(\mo(R))$.

(2) $\Sing(R) \subseteq V(\sqrt{\ca}_{R}(\mo(R))) \subseteq V(\ca(\mo(R))) \subseteq V(\mathfrak{ca}(R))$.
\end{lemma}
\begin{proof}
(1) This follows from Lemma \ref{test regular}.

(2) By (1), it remains to prove $\Sing(R) \subseteq V(\sqrt{\ca}_{R}(\mo(R)))$. Let $\p$ be a prime ideal of $R$ such that $\p\notin V(\sqrt{\ca}_{R}(\mo(R)))$. Since the inclusion relation $V(\sqrt{\ca_{R}(R/\p)})\subseteq V(\sqrt{\ca}_{R}(\mo(R)))$ holds, the prime ideal $\p$ does not belong to $V(\sqrt{\ca_{R}(R/\p)})=V(\ca_{R}(R/\p))$. This implies that $\p\notin\IPD(R/\p)$; see Proposition \ref{relation}. Hence, the $R_{\p}$-module $R_{\p}/\p R_{\p}$ has finite projective dimension. This implies that $\p$ does not belong to $\Sing R$, and hence $\Sing(R) \subseteq V(\sqrt{\ca}_{R}(\mo(R)))$.
\end{proof}
Regarding the generation of the singularity category and the vanishing of the cohomology annihilator, we have the following.
\begin{theorem}\label{main-second}
Let $R$ be a commutative Noetherian ring. Then $\D_{\sg}(R/\p)$ has a generator for each prime ideal $\p$ of $R$ if and only if $\sqrt{\ca}_{R/\p}(\mo(R/\p))\neq 0$ for each prime ideal $\p$ of $R$. 

\end{theorem}
\begin{proof}
For the forward direction, assume that $\D_{\sg}(R/\p)$ has a generator for each prime ideal $\p$ of $R$. To obtain the conclusion, it suffices to show  $\sqrt{\ca}_R(\mo(R)) \neq 0$ under the assumption that $R$ is a domain. By assumption and Proposition \ref{recover} (1), there exists $M\in\mo(R)$ such that $M$ is a generator of $\D_{\sg}(R)$ and $\Sing(R)=V(\ca_R(M))$. With the same argument in the proof of Corollary \ref{nonzero}, $\ca_R(M)\neq 0$, and hence there exists $f\neq 0$ in $\ca^n_R(M)$ for some $n\geq 0$. Combining this with Lemma \ref{basic} (1), we get that the projective dimension of $M_f$ over $R_f$ is finite. Note that $M_f$ is a generator of $\D_{\sg}(R_f)$ as $M$ is a generator of $\D_{\sg}(R)$.
Hence, $R_{f}$ is a regular ring. Let $X$ be a finitely generated $R$-module. Since the $R_{f}$-module $X_{f}$ has finite projective dimension, there exists  $m\geq 0$ such that
$$
\Ext_{R}^{m}(X, \Omega_{R}^{m} (X))_f \cong \Ext_{R_{f}}^{m}(X_{f}, (\Omega_{R}^{m} (X))_{f}) = 0.
$$ 
It follows that $f^{r}\cdot\Ext_{R}^{m}(X, \Omega_{R}^{m} (X)) = 0$ for some $r \geq 0$. This yields that $f\in\sqrt{\ca_{R}^{m}(X)} \subseteq \sqrt{\ca_{R}(X)}$. Thus, 
$$
f \in \displaystyle\bigcap_{X \in \mo(R)} \sqrt{\ca_R(X)} = \sqrt{\ca}_{R}(\mo(R)).
$$

For the backward direction, assume that $\sqrt{\ca}_{R/\p}(\mo(R/\p))\neq 0$ for each prime ideal $\p$ of $R$. Then 
$0 \in \Spec(R/\p)\setminus V(\sqrt{\ca}_{R/\p}(\mo(R/\p))) \subseteq \operatorname{Reg}(R/\p)$ in $\Spec (R/\p)$, where $\Reg(R/\p)=\Spec(R/\p)\setminus \Sing(R/\p)$; the inclusion is from Lemma \ref{syzygy}. This implies that $\Reg(R/\p)$ contains a nonempty open subset. By virtue of \cite[Theorem 1.1]{IT2019}, $\D_{\sg}(R/\p)$ has a generator for each prime ideal $\p$ of $R$.
\end{proof}
We end with this section by showing that there exists a one-dimensional Noetherian domain satisfying $\ann_R\D_{\sg}(R)=0$. 
\begin{example}\label{ann sin-cat zero}
    In \cite[Example 1]{Hochster}, Hochster constructed a one-dimensional Noetherian domain $R$  whose regular locus (i.e., $\Spec(R)\setminus \Sing(R))$ does not contain a nonempty open subset. According to \cite[Theorem 1.1]{IT2019}, there exists a prime ideal $\p$ of $R$ such that $\D_{\sg}(R/\p)$ does not have a generator. This implies that $\p$  must be the zero ideal as $R$ is a one-dimensional domain. Therefore, $\D_{\sg}(R)$ does not have a generator, and by Corollary \ref{main application}, we conclude that $\ann_R\D_{\sg}(R)=0$.
\end{example}
\section{co-cohomological annihilators for modules}\label{Section coca of modules}
In this section, we introduce the notion of a co-cohomological annihilator of modules. The main result is Theorem \ref{T2} from the introduction.
\begin{chunk}\label{def of co ca}
    For each $R$-module $M$ and $n\geq 0$, we define the $n$-th \emph{co-cohomological annihilator} of $M$ to be
    $$
    \coca_R^n(M)\colonequals \ann_R\Ext^{\geq n}_R(\mo(R),M).
    $$
In words, $\coca_R^n(M)$ consists of elements $r\in R$ such that $r\cdot \Ext^i_R(N,M)=0$ for each $i\geq n$ and $N\in\mo(R)$. By the dimension shifting, we have $\coca_R^n(M)=\ann_R\Ext^n_R(\mo(R),M)$. Consider the ascending chain of ideals
$$
\ann_RM=\coca_R^0(M)\subseteq \coca_R^1(M)\subseteq\coca_R^2(M)\subseteq\cdots,
$$
the co-cohomological annihilator of $M$ is defined to be the ideal
$$
\coca_R(M)\colonequals \bigcup_{n\geq 0}\coca_R^n(M).
$$
Since $R$ is Noetherian, $\coca_R(M)=\coca_R^n(M)$ for $n\gg 0$.  

Combining with Baer's criterion for injectivity, we conclude that, for each $R$-module $M$ and $n\geq 0$, $\coca_R^n(M)=R$ if and only if $\id_R(M)<n$.  Thus, $\coca_R(M)=R$ if and only if $\id_R(M)<\infty$.
\end{chunk}

\begin{lemma}\label{basic123}
For each $R$-module $M$ and $n >0$, 
    $$\{\p\in \Spec(R)\mid \id_{R_\p}(M_\p)\geq  n\}
     =\Supp_R\Ext^n_R(\mo(R), M).$$
\end{lemma}
\begin{proof}
Note that $\id_{R_\p}(M_\p)<n$ if and only if $\Ext^n_{R_\p}(\mo(R_\p),M_\p)=0$; this can be proved by using Baer's criterion for injectivity. This yields the first equality below:
    \begin{align*}
        \{\p\in \Spec(R)\mid \id_{R_\p}(M_\p)\geq n\}& =\{\p\in \Spec(R)\mid \Ext_{R_\p}^n(\mo(R_\p), M_\p)\neq 0\}\\
        &=\{\p\in \Spec(R)\mid \Ext_{R}^n(\mo(R), M)_\p\neq 0\}\\
        &=\Supp_R\Ext^n_R(\mo(R),M),
    \end{align*}
    where the second one follows from \ref{localization} and the fact that $\mo(R)\rightarrow \mo(R_\p)$ is dense. This finishes the proof.
\end{proof}

\begin{chunk}\label{def of IID}
  Let $M$ be a finitely generated $R$-module, the \emph{infinite injective dimension locus} of $M$ is defined to be
$$
\IID(M)\colonequals \{\p\in \Spec(R)\mid \id_{R_\p}(M_\p)=\infty\}.
$$ 
\end{chunk}

\begin{proposition}\label{singid}
If $\mo(R)$ has an extension generator $G$, then $\Sing(R)=\IID(G)$.
\end{proposition}

\begin{proof} 
We only need to prove $\Sing(R)\subseteq \IID(G)$, i.e., $\Spec(R)\setminus \IID(G)\subseteq \Spec(R)\setminus \Sing(R)$. By hypothesis, there exists $s\geq 0$ such that $\syz^s_R(\mo(R))\subseteq \bigcup\limits_{t\ge 0}|G|_t$. Pick $\p\in \Spec(R)\setminus \IID(G)$. That is, $\id_{R_{\p}} (G_{\p})<\infty$. Note that $(\syz^s_R (R/\p))_{\p}\in \bigcup\limits_{t\ge 0}|G_{\p}|_t$. Let $\syz^s_{R_{\p}}$ denote a minimal syzygy, then $$R_{\p}^{n}\oplus\syz^s_{R_{\p}}(k(\p))\cong (\syz^s_R(R/\p))_{\p}$$ for some $n\geq 0$, where $k(\p)$ is the residue field of $R_{\p}$. It follows that $\syz^s_{R_{\p}}(k(\p))\in \bigcup\limits_{t\geq 0}|G_{\p}|_t$.  Since $\id_{R_{\p}}(G_{\p})<\infty$, each module in $\bigcup\limits_{t\geq 0}|G_{\p}|_t$ has finite injective dimension over $R_{\p}$. In particular, $\id_{R_{\p}}(\syz^s_{R_{\p}} (k(\p)))<\infty$. Combining this with \cite[Theorem 3.7]{ggp},  we conclude that $R_{\p}$ is regular, i.e., $\p \in \Spec(R)\setminus \Sing(R)$. 
\end{proof}
\begin{lemma}\label{contain12}
    For each $R$-module $M$ and $n >0$, 
    $$\{\p\in \Spec(R)\mid \id_{R_\p}(M_\p)\geq n\}\subseteq V(\coca_R^n(M)).
    $$
    In particular, $\IID(M)\subseteq V(\coca_R(M)).$ Consequently, $$\Sing(R)\subseteq V(\bigcap_{M\in \mo(R)} \sqrt{\coca_R(M)})\subseteq V(\bigcap_{M\in \mo(R)} \coca_R(M))$$
\end{lemma} 
\begin{proof}
By Lemma \ref{basic123}, we have the first equality of the following:
    \begin{align*}
  \{\p\in \Spec(R)\mid \id_{R_\p}(M_\p)\geq  n\}
    & =\Supp_R\Ext^n_R(\mo(R), M)\\
    &=\bigcup_{N\in \mo(R)}\Supp_R\Ext^n_R(N,M)\\
    & \subseteq\bigcup_{N\in \mo(R)}V(\ann_R\Ext^n_R(N,M))\\
    & \subseteq V(\bigcap_{N\in \mo(R)}\ann_R\Ext^n_R(N,M)).
\end{align*}
Combining with $\coca_R^n(M)=\displaystyle\bigcap\limits_{N\in \mo(R)}\ann_R\Ext^n_R(N,M)$, the first statement of the lemma follows.

By the first statement, we get the inclusion below:
\begin{align*}
    \IID(M)&=\bigcap_{n>0}\{\p\in \Spec(R)\mid \id_{R_\p}(M_\p)\geq  n\}\\
    &\subseteq \bigcap_{n>0}V(\coca_R^n(M)).
\end{align*}
Combining with $\displaystyle\bigcap_{n>0}V(\coca_R^n(M))
    =V(\bigcup\limits_{n>0}\coca_R^n(M))
     =V(\coca_R(M))$, we conclude that $\IID(M)\subseteq V(\coca_R(M))$, and hence $\IID(M)\subseteq V(\sqrt{\coca_R(M)})$. This yields the second inclusion below:
     \begin{align*}
     \Sing(R)& \subseteq \bigcup_{M\in \mo(R)}\IID(M) \\
     &\subseteq \bigcup_{M\in \mo(R)} V(\sqrt{\coca_R(M)})\\
     &\subseteq V(\bigcap_{M\in \mo(R)} \sqrt{\coca_R(M)})\\
     &\subseteq V(\bigcap_{M\in \mo(R)} \coca_R(M))    
     \end{align*}
This implies the last statement.  
\end{proof}

The following gives a characterization of the singular locus of $R$ in terms of the co-cohomological annihilator; compare with \cite[Theorem 4.3]{IT2016}. Recall that a category $\mathcal X$ of $\mo(R)$ is said to have finite radius provided that $\mathcal X\subseteq [G]_t^R$ for some $t\geq 0$; see \ref{R/m^i}.

\begin{proposition}\label{singcoca}

 Let $\Omega_R^n (\mo(R))$  has finite radius for some $n\ge 0$. Then $R$ has finite Krull dimension and 
$$\Sing(R)=V(\bigcap_{M\in \mo(R)} \coca_R(M))=V(\bigcap_{M\in \mo(R)}\coca_R^{n+d+1}(M))=V(\CA^{n+d+1}(R))$$ 
for each $d\geq \dim(R)$.

\end{proposition}
 
\begin{proof}  
By \ref{R/m^i}, $\dim(R)$ is finite. The last equality holds because $\CA^{n+d+1}(R)= \bigcap\limits_{M\in \mo(R)}\coca^{n+d+1}_R(M)$. Next, we prove the other equalities.
 By hypothesis, there exists $G\in \mo(R)$ and  $t\geq 0$ such that $\Omega_R^n (\mo(R)) \subseteq [G]_t^R$. 
It follows that $\syz^{n+d}_R(\mo (R))\subseteq [H]_t^R$, where $H:=\syz^d_R(G)$. For each $M\in \mo(R)$, we have 
\begin{align*}
    (\ann_R \Ext^1_R(H, \Omega_R^1(H)))^t& 
     \subseteq (\ann_R \Ext^{\geq 1}_R(H,M))^t\\
    & \subseteq \ann_R\Ext_R^{\geq 1}([H]_t^R,M)\\
   &  \subseteq \ann_R \Ext_R^{\ge n+d+1}(\mo(R), M)\\
  &  \subseteq \coca(M), 
\end{align*}
where the first inclusion is by Lemma \ref{basic}, the second one is due to \cite[Lemma 5.3 (1)]{DT2015}, and the third one is by the inclusion $\syz^{n+d}_R(\mo (R))\subseteq [H]_t^R$. Thus, $$(\ann_R \Ext^1_R(H, \Omega_R^1(H)))^t\subseteq \bigcap_{M\in \mo(R)}\coca^{n+d+1}_R(M)\subseteq \bigcap_{M\in \mo(R)} \coca_R(M).$$
This yields the second and third inclusions below: 
\begin{align*}
    \Sing(R)& \subseteq V(\bigcap_{M\in \mo(R)} \coca_R(M))\\
    & \subseteq V(\bigcap_{M\in \mo(R)}\coca^{n+d+1}_R(M))\\
    & \subseteq V(\ann_R \Ext^1_R(H, \Omega^1_R(H)))\\
    &=\IPD(H),
\end{align*} 
where the first inclusion is by \Cref{contain12}, and the last equality can be deduced by combining the assumption $d\geq \dim(R)$ with \Cref{basic} and \Cref{relation}. The desired equalities now follow by combining with $\IPD(H)\subseteq \Sing(R)$. 
\end{proof}

\begin{chunk}\label{def of ext gen}
   (1) As mentioned in \ref{equi def of strong gen-abel}, $\mo(R)$ has a strong generator if and only if there exist $s,t\geq 0$ and $G\in \mo(R)$ such that
$$
\Omega^s_R(\mo(R))\subseteq |G|_t.
$$

Note that if $\mo(R)$ has finite extension dimension in the sense of Beligiannis \cite[Definition 1.5]{Beligiannis}, then $\mo(R)$ has a strong generator.

(2) Suggested by above, we say that $\mo(R)$ has an \emph{extension generator} if there exists $s\geq 0$ and $G\in \mo(R)$ such that
$$
\Omega^s_R(\mo(R))\subseteq\bigcup_{t\geq 0} |G|_t;
$$
in this case, we say $G$ is an extension generator of $\mo(R)$. By \cite[Remark 2.11]{BHST}, for each $G\in\mo(R)$, the union $\displaystyle\bigcup_{t\geq 0} |G|_t$ is the smallest full subcategory of $\mo(R)$ that contains $G$ and is closed under extensions and direct summands. 

By above, if $\mo(R)$ has a strong generator $G$, then $G$ is an extension generator of $\mo(R)$. 
\end{chunk}

If $\mo(R)$ has an extension generator $G$, then  $\mo(R)$ has a generator $G\oplus R$. However, the next example shows that $\mo(R)$ may not have an extension generator if $\mo(R)$ has a generator. 
\begin{example}\label{generator doesn't imply extension gen}
    Let $R$ be a regular ring with infinite Krull dimension. By \ref{gen regular}, $R$ is a generator of $\mo(R)$. However, $\mo(R)$ doesn't have an extension generator. If not, assume that there exist $s\geq 0$ and $G\in \mo(R)$ such that $\Omega^s_R(\mo(R))\subseteq \displaystyle\bigcup\limits_{t\geq 0}|G|_t$. Since $R$ is regular, there exists $n\geq 0$ such that $\Omega^n_R(G)$ is projective over $R$. This yields the second inclusion below:
    $$
\Omega^{n+s}_R(\mo(R))=\Omega^n_R\Omega^s_R(\mo(R))\subseteq \bigcup_{t\geq 0}|\Omega^n_R(G)|_t\subseteq |R|_1,
$$
where the first inclusion is by the horseshoe lemma. Thus, $\gldim(R)\leq t+s$. By \ref{strong gen regular finite dim}, $R$ has finite Krull dimension. This is a contradiction, and hence $\mo(R)$ doesn't possess an extension generator.
\end{example}

\begin{theorem}\label{equality}
    Let $R$ be a commutative Noetherian ring and $M$ be a finitely generated $R$-module. Then:
\begin{enumerate}
    \item If $\mo(R)$ has an extension generator, then $\IID(M)$ is closed in $\Spec(R)$. 

    \item If $\mo(R)$ has a strong generator, then $$\IID(M)=V(\coca_R(M)).$$
\end{enumerate} 
\end{theorem}

\begin{proof}
(1) Assume $\mo(R)$ has an extension generator $G$. That is, there exists $s\geq 0$ satisfying: for each $X\in \mo(R)$, there exists $t\geq 0$ such that $\Omega^s_R(X)\subseteq |G|_{t}$. For each $n>0$, Lemma \ref{basic123} yields the first equality below:
\begin{align*}
    \{\p\in \Spec(R)\mid \id_{R_\p}(M_\p)\geq n+s\}& =\Supp_R\Ext^{n+s}_R(\mo(R), M)\\
    & =\Supp_R\Ext^{n}_R(\Omega^s_R(\mo(R)), M)\\
    & =\bigcup_{N\in \mo(R)}\Supp_R\Ext^n_R(\Omega_R^s(N),M).
\end{align*}
For each $N\in \mo(R)$, it follows from the assumption that $\Omega^s_R(N)\in |G|_{t}$ for some $t\geq 0$. This yields that
$
\Supp_R\Ext^n_R(\Omega^s_R(N),M)\subseteq \Supp_R \Ext_R^n(G,M).
$
Hence, $$
\{\p\in \Spec(R)\mid \id_{R_\p}(M_\p)\geq n+s\}\subseteq  \Supp_R \Ext_R^n(G,M).
$$
Taking intersections throughout all $n>0$, 
$
\IID(M)\subseteq \displaystyle\bigcap\limits_{n>0}\Supp_R\Ext^n_R(G,M). 
$
On the other hand, by Lemma \ref{basic123}, $$ \bigcap_{n>0}\Supp_R \Ext_R^n(G,M)\subseteq \bigcap_{n>0}\{\p\in \Spec(R)\mid \id_{R_\p}(M_\p)\geq n\}=\IID(M).$$ 
Thus,
$$
\IID(M)=\bigcap\limits_{n>0}\Supp_R\Ext^n_R(G,M).
$$
The desired result follows as $\Supp_R\Ext^n_R(G,M)=V(\ann_R\Ext^n_R(G,M))$ is closed. 

(2) By Lemma \ref{contain12}, it remains to prove $V(\coca_R(M))\subseteq \IID(M)$.  Let $n>0$ be an integer. Since $\mo(R)$ has a strong generator, there exist $s,t\geq 0$ and $G\in \mo(R)$ such that $\Omega^s_R(\mo(R))\subseteq |G|_t$; see \ref{def of ext gen}. Hence, 
    $$
\coca_R^{n+s}(M)=\ann_R\Ext^n_R(\Omega^s_R(\mo(R)),M)\supseteq \ann_R\Ext^n_R(|G|_t, M).
$$
It is routine to check that $\ann_R \Ext^n_R(|G|_t,M)\supseteq (\ann_R\Ext_R^n(G,M))^t $. Thus,
\begin{align*}
    V(\coca_R^{n+s}(M))&\subseteq V(\ann_R\Ext_R^n(G,M))\\
    & \subseteq \{\p\in \Spec(R)\mid \id_{R_\p}(M_\p)\geq n\},
\end{align*}
where the second inclusion is due to $V(\ann_R\Ext^n_R(G,M))=\Supp_R\Ext^n_R(G,M)$ and \ref{localization}; the equality here holds as $\Ext_R^n(G,M)$ is finitely generated. 
Combining with $\coca_R^{n+s}(M)\subseteq \coca_R(M)$, we get
$$
V(\coca_R(M))\subseteq \{\p\in\Spec(R)\mid \id_{R_\p}(M_\p)\geq n\}.
$$
Since $n>0$ is arbitrary, 
$$
V(\coca_R(M))\subseteq \bigcap_{n>0}\{\p\in \Spec(R)\mid \id_{R_\p}(M_\p)\geq n\}=\IID(M).
$$
This completes the proof.
\end{proof}
\begin{remark}
    As noted in Example \ref{generator doesn't imply extension gen}, there exist rings $R$ such that $\mo(R)$ has a generator but does not have an extension generator. Let $R$ be a commutative Noetherian ring and $M\in \mo(R)$. Motivated by Theorem \ref{equality} (1), a natural question arises: is $\IID(M)$ closed if $\mo(R)$ has a generator? 
\end{remark}
\begin{corollary}\label{defining ideal of IID}
    Let $R$ be a quasi-excellent ring with finite Krull dimension. Then 
    $$
    \IID(M)=V(\coca_R(M)).
    $$
\end{corollary}
\begin{proof}
    By Example \ref{quasi-excellent}, $\mo(R)$ has a strong generator. The desired result now follows from Theorem \ref{equality}. 
\end{proof}

\begin{remark}
    (1) Let $R$ be an excellent ring. Greco and Marinari \cite[Corollary 1.5]{Greco-Marinari} observed that $\IID(R)$ is closed. Takahashi \cite{Takahashi:2006_Glasgow} then extended this to modules, proving that $\IID(M)$ is closed for each $M\in\mo(R)$. Recently, Kimura \cite[Theorem 1.1]{Kimura} generalized Takahashi's result to acceptable rings. 
    
    (2) By Proposition \ref{Krull dim finite} and Example \ref{quasi-excellent}, for a quasi-excellent ring $R$, $\mo(R)$ has a strong generator if and only if $R$ has finite Krull dimension. 
\end{remark}
\begin{corollary}
Let $M$ be a finitely generated $R$-module. Assume $\mo(R)$ has a strong generator; equivalently, there exist $s,t\geq 0$ and $G\in \mo(R)$ such that
$
\Omega^s_R(\mo(R))\subseteq |G|_t.
$
 Set $d=\sup \{\id_{R_\p}(M_\p)\mid \p\notin \IID(M)\}$. Then:
 
 (1) $d$ is finite.
 
 (2)
 $
\IID(M)=V(\coca_R(M))=V(\coca_R^{d+s+1}(M)).
$
\end{corollary}
\begin{proof}
(1) This follows from Proposition \ref{Krull dim finite} and \cite[Theorem 3.1.17]{BH}.

  (2) By Theorem \ref{equality}, $\IID(M)=V(\coca_R(M))$.
   It is clear that $$V(\coca_R(M))\subseteq V(\coca_R^{d+s+1}(M)).$$ Assume this is not equal. Combining with $\IID(M)=V(\coca_R(M))$, there exists a prime ideal $\p$ of $R$ containing $\coca_R^{d+s+1}(M)$ and $\p\notin \IID(M)$. Thus, $\p$ contains $\ann_R\Ext^{d+s+1}_R(\mo(R),M)$. By the argument in the proof of Theorem \ref{equality} (2),
  $$ 
    V(\coca_R^{n+s}(M)) \subseteq \{\p\in \Spec(R)\mid \id_{R_\p}(M_\p)\geq n\}
$$
for each $n>0$. Thus, $\id_{R_\p}(M_\p)\geq d+1$. This contradicts with $\p\notin \IID(M)$; by assumption, $\id_{R_\p}(M_\p)\leq d$ if $\p\notin \IID(M)$. This completes the proof.
\end{proof}

\begin{chunk}\label{finite-closed}
(1) Let $V$ be a specialization closed subset of $\Spec(R)$. Namely, $V$ is a union of closed subsets of $\Spec(R)$. 
If, in addition, $V$ is a finite set, then $V$ is a closed subset of $\Spec(R)$ and $\dim(V)\le 1$.  
Indeed, the assumption yields that $V=\displaystyle\bigcup_{\p\in V}V(\p)$. Since $V$ is finite, $V$ is a finite union of closed subsets, and hence $V$ is closed. It follows that $V=V(I)$ for some ideal $I$ of $R$. If $\dim(V)(=\dim(R/I))\geq 2$, then there exists a chain of prime ideals $\p_0\subsetneq \p_1\subsetneq\p_2$ in $V(I)=V$. By \cite[Theorem 144]{kap}, $V$ is infinite. This is a contradiction. Hence, $\dim(V)\le 1$. 

(2) If $V$ is a closed subset of $\Spec(R)$, $R$ is semi-local, and $\dim(V)\leq 1$, then $V$ is a finite set. 
In fact, the assumption yields that $V=V(I)$ for some ideal $I$ of $R$. Since $R/I$ is semi-local and $\dim(R/I)=\dim(V)\leq 1$, one has $\Spec(R/I)$ is a finite set. It follows that $V(I)=V$ is finite.
\end{chunk}
\begin{example}\label{Example-finite}
    For a commutative Noetherian ring $R$, the singular locus $\Sing(R)$ is finite  in any of the following cases: 

\begin{enumerate}
    \item $R$ is a semi-local ring with $\dim (R)= 1$. 

    \item $R$ is semi-local with isolated singularities. 

    \item $R$ is a semi-local J-$0$ 
 domain with $\dim (R)=2$. 

    \item $R$ is a semi-local J-$1$ normal ring with $\dim (R)=3$. 
\end{enumerate}

For (1), since $\dim(R)= 1$, $R$ is semi-local, and the set of minimal prime ideals of  $R$  is finite,  it follows that \(\Sing(R)\) is finite. (2) holds immediately by assumption. 

For (3), since $R$ is J-$0$, we have $\Sing(R) \subseteq V(I)$ for some nonzero ideal $I$ of $R$, and as $R$ is a domain, it follows that $\dim(R/I)\leq 1$. Combining this with that $R$ is semi-local, we conclude that $V(I)=\Spec(R/I)$ is finite, and hence $\Sing(R)$ is finite. 

For (4), since $R$ is J-$1$,  we have $\Sing(R)=V(I)$ for some ideal $I$ of $R$, and as $R$ is normal, $V(I)$ cannot contain any prime ideal of height $0$ or height $1$. Thus, the height of $I$ is at least $2$. It follows that $\dim(R/I) \leq 1$,  and by the same reasoning as in (3), \( \Sing(R) \) is finite.
\end{example}
 \begin{chunk} Let $V$ be a closed subset of $\Spec(R)$. We define the \emph{arithmetic rank} of $V$ to be
 $$\ara(V):=\inf\{n\geq 0\mid V=V(x_1,\cdots,x_n)\text{ for some  } x_1,\ldots,x_n\in R\}.$$ 
 If we write $V=V(I)$ for some ideal $I$ of $R$, then the value $\ara(V(I))$ coincides with the arithmetic rank of $I$ as defined in \cite[Definition 3.3.2]{BS2012}.
\end{chunk} 
\begin{proposition}\label{finite-extension gen}
      Let $R$ be a commutative Noetherian ring with finite Krull dimension $d$.
  Assume $\Sing(R)$ is a finite set. Then $\Sing(R)$ is closed. Moreover,   
       $$\Omega^d_R(\mo(R))\subseteq \bigcup_{t\geq 0}\Big|\bigoplus_{i=0}^n \bigoplus_{\p\in \Sing(R)} \Omega^i_R(R/\p)\Big|_t,$$ 
     where $n:=\ara (\Sing(R))$. In particular, $\mo(R)$ has an extension generator.
\end{proposition}
\begin{proof}
    By \ref{finite-closed}, $\Sing(R)$ is closed. Assume $\Sing(R)=V(x_1,\ldots,x_n)$ for some elements $x_1,\ldots,x_n$ in $R$. Let $M$ be a finitely generated $R$-module. For each $\p\notin \Sing(R)$, we have
    $$
    \Ext^1_R(\Omega^d_R(M),\Omega^{d+1}_R(M))_\p=0.
    $$
    It follows that the non-free locus of $\Omega^d_R(M)$, given by $\{\p\in \Spec(R)\mid \Omega^d_R(M)_\p \text{ is not free}\}$,  is contained in $\Sing(R)$. By \cite[Theorem 3.1 (2)]{BHST}, $\Omega^d_R(M)\in \displaystyle\bigcup_{t\geq 0}\Big|\bigoplus_{i=0}^n \bigoplus_{\p\in \Sing(R)} \Omega^i_R(R/\p)\Big|_t$.
\end{proof}
\begin{remark}\label{AIT}
 Proposition \ref{finite-extension gen} extends a recent result of Araya, Iima, and Takahashi \cite[Proposition 4.7 (1)(2)]{AIT}.  Specifically, if $R$ is either a semi-local ring with isolated singularities or a semi-local J-$0$ domain with $\dim(R) = 2$, then $\Sing(R)$ is finite by Example \ref{Example-finite} (2) and (3), and hence $\mo(R)$ has an extension generator by Proposition \ref{finite-extension gen}. 
\end{remark}

The following result concerns the relationship between the generator of the singularity category and the extension generator of the module category; compare Example \ref{generator doesn't imply extension gen}.
\begin{proposition}\label{gen-extension gen}
    Let $R$ be a commutative Noetherian ring with finite Krull dimension $d$. If $\D_{\sg}(R/\mathfrak{p})$ has a generator for each prime ideal $\mathfrak{p} \in \Spec(R)$, then $\Omega^d_R(\mo(R)) \subseteq \displaystyle\bigcup_{t\geq 0} |G|_t$ for some $G \in \mo(R)$. In particular, $\mo(R)$ has an extension generator.
\end{proposition}
 \begin{proof}
      We prove it by induction on $d = \dim R$. Note that when $d = 0$, $R/\mathfrak{p}$ is  Artinian for every $\mathfrak{p}$, and hence $\mo(R/\mathfrak{p})$ has an extension generator. So we assume $d > 0$.

First, we consider the case when $R$ is a domain. By \cite[Theorem 1.1]{IT2019}, $\Sing(R)=V(I)$ for some ideal $I$. As $R$ is a domain, $I$ is a non-zero ideal, so pick $0 \neq x \in I$. Then $R_x$ is regular. As every prime ideal of $R/(x)$ is of the form $\mathfrak{p}/(x)$ for some $\mathfrak{p} \in \Spec(R)$, and $(R/(x))/(\mathfrak{p}/(x)) \cong R/\mathfrak{p}$, thus $\D_{\sg} ((R/(x))/{\mathfrak{q}})$ has a classical generator for each $\mathfrak{q} \in \Spec(R/(x))$. Since $R$ is a domain, $n \colonequals \dim (R/(x)) < \dim(R) = d$. Thus, by induction hypothesis, $\Omega^n_{R/(x)}(\mo(R/(x)) \subseteq \displaystyle\bigcup_{t\geq 0} |G|_t^{R/(x)}$ for some $G \in \mo(R/(x))$. By \cite[Lemma 4.5(2)]{AIT}, there exists $H \in \operatorname{mod}(R)$ such that $\Omega^d_R(\mo(R)) \subseteq \displaystyle\bigcup_{t\geq 0} |H|_t^R$. This completes the case where $R$ is a domain.

The proof of the general case now proceeds similarly to part (b) of the proof in \cite[Proposition 4.6(3)]{AIT}; note that each $R/\mathfrak{p}_i$ in loc. cit. is a domain and satisfies the same hypothesis on the singularity category.
\end{proof}
The following result is related with Proposition \ref{Krull dim finite}; $\Rfd$ represents the large restricted flat dimension, see \ref{Rfd}. Recall that a full subcategory of $\mo(R)$ is called \emph{resolving} if it contains all finitely generated projective modules and is closed under syzygies, extensions, and direct summands.  For each $M\in \mo(R)$, let $\res(M)$ denote the smallest resolving subcategory of $\mo(R)$ that contains $M$. Note that $|M|_t\subseteq \res(M)$ for each $t\geq 0$.
\begin{proposition}\label{ext-dimension}
    Let R be a commutative Noetherian ring. If there exist an integer $s\geq 0$ and a nonzero finitely generated $R$-module $G$ such that $\Omega_R^s(\mo(R)) \subseteq \res(G)$, then R
has finite Krull dimension and $$\dim(R) \leq  s+\Rfd(G)+1.$$
In particular, if
$\mo(R)$ has an extension generator, then R has finite Krull dimension.
\end{proposition} 
\begin{proof}
     By \ref{Rfd}, $\Rfd(G)$ is finite. Set $\Rfd(G)=n$. By \cite[Definition 2.1 and Theorem 2.4]{CFF}, we conclude that the category $\mathcal{C}\colonequals\{M \in \mo (R) | \Rfd(M) \leq n\}$ is a resolving subcategory of $\mo(R)$
containing $G$, hence it contains $\res(G)$. It follows from the hypothesis that $\Omega_R^s(M)\in \mathcal{C}$ for each $M \in \mo(R)$. Thus, for each $M \in
\mo(R)$, $\Rfd(\Omega_R^s (M)) \leq  n$. Again, by \cite[Definition 2.1 and Theorem 2.4]{CFF},
we conclude that $\Rfd(M) \leq s+n$ for each $M \in \mo(R)$. Since $\depth
(R_\p)\leq \Rfd(R/\p)$, we get $\sup\{\depth(R_\p) \mid \p \in \Spec(R)\} \leq s+n$. Since prime ideals
of localization are localization of prime ideals,  \cite[Lemma 1.4]{CFF} implies
that $\dim(R) -1 \leq s+n$.
\end{proof}

In the following, we say $\mo(R)$ has a \emph{resolving generator} provided that $\Omega^s_R(\mo(R))\subseteq \res(G)$ for some $s\geq 0$ and $G\in\mo(R)$.
\begin{corollary}\label{equi-resolving}
    Let $R$ be a commutative Noetherian ring. Then the following are equivalent.

\begin{enumerate}
    \item  $R$ has finite Krull dimension and $\D_{\sg}(R/\p)$ has a generator for each prime ideal $\p$ of $R$. 

    \item $\mo(R/\p)$ has an extension generator for each prime ideal $\p$ of $R$. 

    \item $\mo(R/\p)$ has a resolving generator for each prime ideal $\p$ of $R$. 
\end{enumerate}
  Moreover, if any of the above equivalent conditions holds, then
$\mo(R/I)$ has an extension generator for each ideal $I$ of $R$.  
\end{corollary}
\begin{proof}
    The second statement and the implication $(1)\Rightarrow (2)$ follows from Proposition \ref{gen-extension gen}. The implication $(2)\Rightarrow (3)$ follows from the definition. 

$(3)\Rightarrow (1)$: Assume $\mo(R/\p)$ has a resolving generator for each $\p\in\Spec(R)$. By Proposition \ref{ext-dimension}, $R/\p$ has finite Krull dimension for each prime ideal $\p$ of $R$. Since the set of minimal prime ideals of $R$ is finite, we conclude that $R$ has finite Krull dimension. 

Fix a prime ideal $\p$ of $R$, assume $G$ is a resolving generator of $\mo(R/\p)$. Note that the modules in $\res(G)$ are contained in $\thick_{\D_{\sg}(R/\p)}(G)$. Since each complex in $\D_{\sg}(R/\p)$ is isomorphic to some shift of a module $M$ in $\mo(R/\p)$,  we conclude that $G$ is a generator of $\D_{\sg}(R/\p)$. This finishes the proof.
\end{proof}

We end this section by comparing the cohomological annihilators and the co-cohomological annihilators; see Proposition \ref{compare ca and caco}. 
Recall that $R$ is said to be \emph{Gorenstein in codimension n} if
 $R_\p$ is Gorenstein for each prime ideal $\p$ of $R$ with height at most $n$.
\begin{lemma}\label{example of co ca}
(1) Let $(R,\m)$ be a $d$-dimensional Cohen--Macaulay local ring with a canonical module $\omega$. Set $(-)^\dagger=\Hom_R(-,\omega)$. For each $M\in \CM(R)$, $\ca_R^1(M^{\dagger})\subseteq \coca^{d+1}_R(M)$, and equality holds if, in addition, $R$ is Gorenstein in codimension $d-1$. 

(2) Let $R$ be a Gorenstein ring with finite Krull dimension $d$.  For each $M\in\CM(R)$, $$\coca_R^{d+1}(M)=\coca_R(M)=\ann_R\underline{\End}_R(M)=\ca_R(M).$$

(3) Let R be a commutative Noetherian ring. Then $R$ is Gorenstein with finite Krull dimension if and only if  
$\ca_R(M)=\coca_R(M)$ for each $M\in \mo(R)$ with $\Rfd_R(M)=0.$
\end{lemma}
\begin{proof}
(1) For each $X\in \mo(R)$, since $\Omega^d_R(X)\in \CM(R)$, we get $$\ca^1_R(M^{\dagger})\subseteq \ann_R \Ext^1_R(M^{\dagger}, (\Omega_R^d (X))^{\dagger})=\ann_R \Ext^1_R(\Omega_R^d(X), M)=\ann_R \Ext^{d+1}_R(X,M).$$ 
Since this is true for any $X\in \mo(R)$, we get $\ca_R^1(M)\subseteq \ann_R \Ext^{d+1}_R(\mo(R), M)=\coca_R^{d+1}(M)$.

Assume $R$ is Gorenstein in codimension $d-1$. Note that $(\Omega_R^1(M^{\dagger}))^{\dagger}\in \CM(R)$. By \cite[Theorem 3.8]{EG}, $(\Omega_R^1(M^{\dagger}))^{\dagger}\cong \Omega^d_R(Y)$ for some $Y\in \mo(R)$. This yields the first equality below:
$$\coca_R^{d+1}(M)\subseteq \ann_R \Ext^{d+1}_R(Y, M)=\ann_R \Ext^1_R((\Omega_R^1(M^{\dagger}))^{\dagger},M)=\ann_R \Ext^1_R(M^{\dagger}, \Omega_R^1(M^{\dagger}))=\ca_R^1(M^\dagger),$$
where the last equality is by Lemma \ref{basic}. This proves the converse direction.

 (2) For each $n>0$, assume $r\in \coca_R^n(M)$. This yields 
$$
0=r\cdot \Ext^n_R(\mo(R), M)=r\cdot \Ext^1_R(\Omega_R^{n-1}(\mo(R)),M).
$$
Since $M\in \CM(R)$ and $R$ is strongly Gorenstein (see \ref{def of Gorenstein}), $M$ is an infinite syzygy. In particular, there is a short exact sequence
$
\xi\colon 0\rightarrow M\rightarrow P\rightarrow C\rightarrow 0,
$
  where $P$ is a finitely generated projective and $C\in \Omega_R^{n-1}(\mo(R))$.  By above, $r\cdot \xi=0$. In particular, $r\colon M\rightarrow M$ factors through $P$. This implies that $r\in \ann_R\underline{\End}_R(M)$, and therefore $\coca_R^n(M)\subseteq\ann_R\underline{\End}_R(M)$ for each $n>0$. It follows from this that $\coca_R(M)\subseteq \ann_R\underline{\End}_R(M)$. 

  By above, $\coca_R^{d+1}(M)\subseteq \coca_R(M)\subseteq \ann_R\underline{\End}_R(M)=\ca_R(M)$, where the equality here follows from Lemma \ref{basic} and Proposition \ref{relation CA(R) and ca(R-mo)}. Next, we prove the inclusion $\ann_R\underline{\End}_R(M)\subseteq \coca_R^{d+1}(M)$.  Assume $a\in \ann_R\underline{\End}_R(M)$. Namely, there is a factorization
$$
\xymatrix{
M\ar[rr]^-a\ar[rd]&& M\\
& Q\ar[ru]& 
},
$$
where $Q$ is finitely generated projective. Note that $\id_R(Q)\leq d$; see \ref{def of Gorenstein}. Hence, $\Ext^{d+1}_R(X,Q)=0$ for each $X\in \mo(R)$. Applying $\Ext^{d+1}_R(X,-)$ to the above diagram, we conclude that $a\colon \Ext^{d+1}_R(X,M)\rightarrow \Ext^{d+1}_R(X,M)$ is zero for each $X\in \mo(R)$. That is, $a\in \coca_R^{d+1}(M)$, and hence $\ann_R\underline{\End}_R(M)\subseteq \coca_R^{d+1}(M)$. Finally, $\ann_R\underline{\End}_R(M)=\ca_R(M)$ follows from \Cref{relation CA(R) and ca(R-mo)} (1).

(3) For the forward direction, assume that $R$ is Gorenstein with finite Krull dimension. For each $M\in\mo(R)$,  if $\Rfd(M)=0$, then $M$ is maximal Cohen--Macaulay. By (2), this yields that $\ca_R(M)=\coca_R(M)$. 

For the backward direction, assume that $\ca_R(M)=\coca_R(M)$ for each $M\in\mo(R)$ with $\Rfd_R(M)=0$. In particular, since $\Rfd_R(R)=0$, we have $\ca_R(R)=\coca_R(R)$. It follows that $\coca_R(R)=R$ as $\ca_R(R)=R$. Hence, there exists $n\geq 0$ such that $\coca^n_R(R)=R$. In particular, $\Ext^j_R(R/\p, R)=0$ for all $j \geq n$ and $\p\in \Spec(R)$. Localizing at $\p$, we conclude that $\id_{R_\p}(R_\p)\leq n$ for each $\p\in\Spec(R)$, and hence $\id_R(R)\leq n$. This shows that $R$ is Gorenstein with finite Krull dimension; see \ref{def of Gorenstein}.
\end{proof}

\begin{proposition}\label{compare ca and caco}
    Let $R$ be a commutative Noetherian ring and $M$ be a finitely generated $R$-module. Assume $M$ is an infinite syzygy, then:

\begin{enumerate}
    \item $\coca_R(M)\subseteq \ca_R^1(M)=\ann_R\underline{\End}_R(M)\subseteq \ca_R(M)$. All these inclusions are equal if, in addition, $R$ is Gorenstein with finite Krull dimension. 

    \item If $\mo(R)$ has a strong generator, then 
    $$
    \IPD(M)=V(\ca_R(M))\subseteq V(\coca_R(M))=\IID(M). 
    $$
\end{enumerate}
\end{proposition}
\begin{proof}
    (1) By Lemma \ref{basic}, it remains to prove $\coca_R(M)\subseteq \ann_R\underline{\End}_R(M)$. This follows from the same argument as in the proof of Lemma \ref{example of co ca} (2). If, in addition, $R$ is strongly Gorenstein, then $M$ is in $\CM(R)$. Hence, by Lemma \ref{example of co ca} (2),  we have $\coca_R(M)=\ca_R(M)$.

  (2) The statement follows from Proposition \ref{relation}, Theorem \ref{equality}, and (1). 
\end{proof}

\begin{proposition}\label{cocacm}
Let $R$ be a Cohen--Macaulay ring with finite Krull dimension $d$. Assume $R$ has a canonical module. Then $$\bigcap_{M\in \mo(R)}\coca_R^i(M)=\bigcap_{N\in \CM (R)}\coca_R^i(N)$$ for all $i\geq \dim(R)+1$ and $\bigcap\limits_{M\in \mo(R)}\coca_R(M)= \bigcap\limits_{N\in \CM (R)}\coca_R(N)$.
    
\end{proposition}

\begin{proof} 
It is clear that $$\bigcap\limits_{M\in \mo(R)}\coca_R^i(M)\subseteq \bigcap\limits_{N\in \CM (R)}\coca_R^i(N) \text{ and } \bigcap\limits_{M\in \mo(R)}\coca_R(M)\subseteq \bigcap\limits_{N\in \CM (R)}\coca_R(N).$$
It remains to prove the converse inclusions.

Let $\omega$ be a dualizing module of $R$.
Since $R\in \CM(R)$ and $\syz^d_R(\mo(R))\subseteq \CM(R)$, each module in $\mo(R)$ has finite $\CM(R)$-resolution dimension in the sense of \cite[Section 1]{buchw}. Moreover, the subcategory $\add(\omega$) is a cogenerator for $\CM(R)$ in the sense of \cite[Section 1]{buchw}. By taking $\mathbf X:=\CM(R)$ in  \cite[Theorem 1.1]{buchw}, for each $M\in \mo (R)$, there exists $X_M\in \CM(R)$ and $Y_M\in \mo(R)$ which fits in an exact sequence $$0\to Y_M\to X_M\to M\to 0,$$ where $Y_M$ has finite $\add(\omega)$-resolution. It follows that $\id_R(Y_M)<\infty$, and hence $\id_R (Y_M)\leq d$; see \cite[Theorem 3.1.17]{BH}.  This yields that $\Ext_R^{\ge d+1}(\mo(R), Y_M)=0$, and hence we get that $\ann_R \Ext^i_R(L, X_M)=\ann_R\Ext^i_R(L,M)$ for $i\ge d+1$ and $L\in \mo(R)$. Thus, for each $M\in \mo(R)$, there exists $X_M\in \CM(R)$ such that $\coca_R^i(X_M)=\coca_R^i(M)$ for all $i\ge d+1$. In particular, $\coca_R(X_M)=\coca_R(M)$.  This yields that $\bigcap\limits_{M\in \mo(R)}\coca_R^i(M)\supseteq \bigcap\limits_{N\in \CM (R)}\coca_R^i(N)$ for $i\ge d+1$ and $\bigcap\limits_{M\in \mo(R)}\coca_R(M)\supseteq \bigcap\limits_{N\in \CM (R)}\coca_R(N)$. 
\end{proof}

\begin{corollary} 
Let $R$ be a Gorenstein ring with finite Krull dimension $d$. For each $i\ge d+1$,  
$$\bigcap_{M\in \mo(R)} \coca^i_R(M)=\bigcap_{M\in \mo(R)} \coca_R(M)=\ann_R \D_{\sg}(R)$$ 
If, in addition, $\D_{\sg}(R)$ has a strong generator, then all the above ideals define $\Sing(R)$. 
\end{corollary}

\begin{proof}
The second statement is a consequence of the first and \Cref{recover} (2). 
For the first one, since $R$ itself is a canonical module, by \Cref{cocacm} it is enough to prove $$\bigcap_{M\in \CM (R)} \coca^i_R(M)=\bigcap_{M\in \CM (R)} \coca_R(M)=\ann_R \D_{\sg}(R)$$  
By \Cref{example of co ca} (2), we have $\coca_R^i(M)=\coca_R(M)=\ann_{\D_{\sg}(R)}(M)$ for each $M\in \CM(R)$ and $i\geq d+1$.  The desired equalities follow by combining this with $\ann_R\D_{\sg}(R)=\bigcap\limits_{M\in \CM(R)} \ann_{\D_{\sg}(R)}(M)$.
\end{proof}

\bibliographystyle{amsplain}
\bibliography{ref}
\end{document}